\documentclass[mathpazo]{cicp}

\usepackage{amssymb}
\usepackage{bbm}
\usepackage{bm}
\usepackage[a4paper, margin=30mm]{geometry}
\usepackage{tikz}

\newcommand{\bx}{\ensuremath{\boldsymbol{x}}}
\newcommand{\by}{\ensuremath{\boldsymbol{y}}}

\newcommand{\bi}{\ensuremath{\boldsymbol{i}}}

\newcommand{\comment}[1]{}
\renewcommand{\arraystretch}{1.2}
\newcommand{\veps}{\varepsilon}
\newcommand{\bbN}{\mathbb{N}}
\newcommand{\bbR}{\mathbb{R}}

\begin{document}
\title{Better Approximations of High Dimensional Smooth
	Functions by Deep Neural Networks with Rectified Power
	Units}

\author[B.~Li, S.~Tang \& H.~Yu]{Bo Li\affil{2}\comma\affil{1}$^{,\dagger}$, Shanshan Tang\affil{3}$^{,\dagger,\ddagger}$ and
		Haijun Yu\affil{1}\comma\affil{2}\comma\corrauth
		}
\footnotetext{$^\dagger$\ The first two authors contributed equally. Author list is alphabetical.}
\footnotetext{$^\ddagger$\ The work of this author is partially done during her Ph.D. study in Academy of Mathematics and Systems Science, Chinese Academy of Sciences.}	
\address{\affilnum{1}\ NCMIS \& LSEC, Institute of Computational Mathematics and Scientific/Engineering Computing, 
	Academy of Mathematics and Systems Science, Chinese Academy of\\ Sciences, Beijing 100190, China\\
	\affilnum{2}\ 
		School of Mathematical Sciences, University of Chinese Academy of Sciences, \\ Beijing 100049, China\\
	\affilnum{3}\ 
	China Justice Big Data Institute, Beijing 100043, China}
%\emails{{\tt libo1171309228@lsec.cc.ac.cn} (B.~Li),\\
%		{\tt tangshanshan@lsec.cc.ac.cn} (S.~Tang),  
%		{\tt hyu@lsec.cc.ac.cn} (H.~Yu)
%}

\email{ {\tt libo1171309228@lsec.cc.ac.cn(B.~Li),\\ tangshanshan@lsec.cc.ac.cn(S.~Tang), hyu@lsec.cc.ac.cn(H.~Yu) } }

%%%%% Begin Abstract %%%%%%%%%%%
\begin{abstract}
  Deep neural networks with rectified linear units (ReLU)
are getting more and more popular due to their universal
representation power and successful applications. Some
theoretical progress regarding the approximation power
of deep ReLU network for functions
in Sobolev space and Korobov space have recently been
made by [D. Yarotsky, Neural Network, 94:103-114, 2017]
and [H. Montanelli and Q. Du, SIAM J Math. Data Sci.,
1:78-92, 2019], etc. In this paper, we show that deep
networks with rectified power units (RePU) can give
better approximations for smooth functions than deep
ReLU networks. Our analysis bases on classical
polynomial approximation theory and some efficient
algorithms proposed in this paper to convert polynomials
into deep RePU networks of optimal size with no
approximation error. Comparing to the results on ReLU
networks, the sizes of RePU networks required to
approximate functions in Sobolev space and Korobov space
with an error tolerance $\varepsilon$, by our
constructive proofs, are in general $\mathcal{O}(\log
\frac{1}{\varepsilon})$ times smaller than the sizes of
corresponding ReLU networks constructed in most of the
existing literature.  Comparing to the classical results
of Mhaskar [Mhaskar, Adv. Comput. Math. 1:61-80, 1993], 
our constructions use less number of
activation functions and numerically more stable, 
they can be served as good initials of deep RePU networks 
and further trained to break the limit of linear approximation theory. 
The functions represented by RePU networks are smooth
functions, so they naturally fit in the places where
derivatives are involved in the loss function.
\end{abstract}


%%%% maketitle %%%%%
\maketitle

\section{Introduction}

Artificial neural network(ANN), whose origin may date back
to the 1940s\cite{mcculloch_logical_1943}, is one of the
most powerful tools in the field of machine learning.
Especially, it became dominant in a lot of applications
after the seminal works by Hinton et
al.\cite{hinton_fast_2006} and Bengio et
al.\cite{bengio_greedy_2007} on efficient training of deep
neural networks (DNNs), which pack up multi-layers of units
with some nonlinear activation function.  Since then, DNNs
have greatly boosted the developments in different areas including image
classification, speech recognition, computational chemistry
and numerical solutions of high-dimensional partial
differential equations and scientific problems, etc., see
e.g. \cite{krizhevsky_imagenet_2012}\cite{hinton_deep_2012}
\cite{lecun_deep_2015}\cite{han_deep_2018}\cite{han_solving_2018}\cite{zhang_deep_2018}\cite{strofer_datadriven_2019}\cite{ma_model_2019}\cite{cavaglia_finding_2019}
to name a few.

The success of DNNs relies on two facts: 1) DNN is a
powerful tool for general function approximation; 2)
Efficient training methods are available to find minimizers
with good generalization ability.  In this paper, we focus
on the first fact.  It is known that artificial neural
networks can approximate any $C^0$ and $L^1$ functions with
any given error tolerance, using only one hidden layer (see
e.g. \cite{cybenko_approximation_1989}\cite{mhaskar_neural_1996}).
However, it was realized recently that deep networks have
better representation power( see
e.g. \cite{delalleau_shallow_2011}\cite{telgarsky_representation_2015}\cite{eldan_power_2016})
than shallow networks.  One of the commonly used activation
functions with DNN is the so called rectified linear unit
(ReLU)\cite{glorot_deep_2011}, which is defined as
$\sigma(x)=\max(0,x)$.  Telgarsky
\cite{telgarsky_representation_2015} gave a simple and
elegant construction showing that for any $k$, there exist
$k$-layer, $\mathcal{O}(1)$ wide ReLU networks on
one-dimensional data, which can express a sawtooth function
on $[0,1]$ with $\mathcal{O}(2^k)$ oscillations. Moreover,
such a rapidly oscillating function cannot be approximated
by poly$(k)$-wide ReLU networks with $o(k/\log(k))$ depth.
Following this approach, several other works proved that
deep ReLU networks have better approximation power than
shallow ReLU networks
\cite{liang_why_2016}\cite{telgarsky_benefits_2016a}\cite{yarotsky_error_2017}\cite{petersen_optimal_2018}.
In particular, for $C^\beta$-differentiable $d$-dimensional
functions, Yarotsky \cite{yarotsky_error_2017} proved that
the number of parameters needed to achieve an error
tolerance of $\varepsilon$ is
$\mathcal{O}(\varepsilon^{-\frac{d}{\beta}}\log\frac{1}{\varepsilon})$.
Petersen and Voigtlaender \cite{petersen_optimal_2018}
proved that for a class of $d$-dimensional piecewise
$C^\beta$ continuous functions with the discontinuous
interfaces being $C^\beta$ continuous also, one can
construct a ReLU neural network with
$\mathcal{O}((1+\frac{\beta}{d})\log_2 (2+\beta) )$ layers,
$\mathcal{O}(\varepsilon^{-\frac{2(d-1)}{\beta}})$ nonzero
weights to achieve $\varepsilon$-approximation. The
complexity bound is sharp. For analytic functions, E and
Wang \cite{e_exponential_2018} proved that using ReLU
networks with fixed width $d+4$, to achieve an error
tolerance of $\varepsilon$, the depth of the network depends
on $\log\frac{1}{\varepsilon}$ instead of $\varepsilon$
itself. We also want to mention that the detailed
relations between ReLU networks and linear finite elements
have been studied by He et al.\cite{he_relu_2018}. And recent
work by Opschoor, Peterson and Schwab \cite{opschoor_deep_2019} reveals the
connection between ReLU DNNs and high-order finite element
methods.

%\newpage
One basic fact on deep ReLU networks is that function $x^2$ can
be approximated within any error $\varepsilon>0$ by a ReLU
network having the depth, the number of weights and
computation units all of order
$\mathcal{O}(\log\frac{1}{\varepsilon})$. This fact has been
used by several groups (see e.g.
\cite{liang_why_2016}\cite{yarotsky_error_2017}) to analyze
the approximation property of general smooth functions using
ReLU networks.  In this paper, we extend the analysis to
deep neural networks using rectified power units (RePUs),
which are defined as
\begin{equation} \label{eq:RePU}
\sigma_s(x) =
\begin{cases}
x^s, & x \ge 0, \\ 0, & x<0,
\end{cases},
\quad 
s\in\bbN_0,
\end{equation}
where $\bbN_0$ denotes the set of non-negative integers.
Note that $\sigma_1$ is the commonly used ReLU function,
$\sigma_0$ is the binary step function.  We call $\sigma_2$,
$\sigma_3$ rectified quadratic unit (ReQU) and rectified
cubic unit, respectively.  We show that deep neural
networks using RePUs($s\ge 2$) as activation functions have
better approximation property for smooth functions than
those using ReLUs.  By replacing ReLU with RePU($s\ge 2$), the
functions $x$, $x^2$ and $xy$ can be exactly represented
with no approximation error using networks having just a few
nodes and nonzero weights. Based on this, we build efficient
algorithms to explicitly convert functions from a polynomial
space into RePU networks having approximately the same
number of coefficients. This allows us to obtain a better
upper bound of the best neural network approximation for
general smooth functions using classical polynomial
approximation theories.  Note that $\sigma_s$ networks have
been used in the classic works by Mhaskar and his coworkers
(see e.g. \cite{mhaskar_approximation_1993}
\cite{chui_neural_1994a}\cite{chuiDeepNetsLocal2018}), where by converting spline
approximations into $\sigma_s$ DNNs, quasi-optimal
theoretical upper bounds of function approximation are
obtained. However, their constructions of neural network are
not optimal for very smooth functions (the case $k\gg s$),
the error bound obtained is quasi-optimal due to an
extra $\log_s(k)$ factor, where $k$ is related to the
smoothness of the underlying functions.  Meanwhile no
numerically efficient and stable algorithm is presented. In
this paper, we present numerically stable and efficient
constructions of RePU network representation of polynomials
which result in RePU network of different structure and
remove the extra $\log_s(k)$ factor in the approximation
bounds. After this paper is put on arXiv, the RePU networks and our optimal network constructions are adopted by other authors, e.g., by using RePU networks instead of ReLU networks, a sharper bound for approximating holomorphic maps in
high dimension is obtained by Opschoor, Schwab and Zech \cite{OSZ19_839}.

For high dimensional problems, to be tractable, the
intrinsic dimension usually do not grow as fast as the
observation dimension. In other words, the problems have low
dimensional structure. A particular example is the class of
high-dimensional smooth functions with bounded mixed
derivatives, for which sparse grid (or hyperbolic cross)
approximation is a very popular approximation tool
\cite{smolyak_quadrature_1963}\cite{bungartz_sparse_2004}\cite{griebel_sparse_2007}\cite{shen_sparse_2010}\cite{dung_highdimensional_2018}. In
the past few decades, sparse grid method and hyperbolic
cross approximations have found many applications, such as
numerical integration and interpolation
\cite{smolyak_quadrature_1963}\cite{gerstner_numerical_1998}\cite{barthelmann_high_2000},\cite{shen_approximations_2014},
solving partial differential equations (PDE)
\cite{bungartz_adaptive_1992} \cite{lin_sparse_2001}
\cite{shen_efficient_2010}\cite{shen_efficient_2012}\cite{wang_sparse_2016}\cite{rong_nodal_2017},
computational chemistry \cite{griebel_sparse_2007}
\cite{yserentant_hyperbolic_2007}\cite{avila_solving_2013}\cite{shen_efficient_2016},
uncertainty quantification
\cite{schwab_sparse_2003}\cite{nobile_sparse_2008}\cite{nobile_adaptive_2016},
etc. For high dimensional problem, we will derive upper
bounds of RePU DNN approximation error by converting sparse
grid and hyperbolic cross spectral approximation into RePU
networks.  Our work is inspired by the recent work of
Montanelli and Du \cite{montanelli_deep_2017}, where the
connection between linear finite element sparse grids and
deep ReLU neural networks is established.  In this paper, we
approximate multivariate functions in high order Korobov
space using sparse grid Chebyshev interpolation
\cite{barthelmann_high_2000} for the interpolation problem,
and using hyperbolic cross spectral approximation for the
projection problem
\cite{shen_sparse_2010}\cite{shen_efficient_2010}. Then, we
convert the high-dimensional polynomial approximations into
ReQU networks, instead of ReLU networks, to avoid adding an
extra factor $\log\frac{1}{\varepsilon}$ in the size of the
neural network.

In summary, we find that RePU networks have the
following good properties:
\begin{itemize}
	\item RePU neural networks provide better approximations for
	sufficient smooth functions comparing to ReLU neural
	network approximations. To achieve same accuracy, the RePU
	network approximation we constructed needs less number of
	layers and smaller network size than existing ReLU neural
	network approximations. For example, for a function with
	all the partial derivatives bounded uniformly independent
	of derivative order, we can construct a ReQU network with
	no more than
	$\mathcal{O}\left(\log_2\left(\log\frac{1}{\varepsilon}\right)\right)$
	layers, and no more than {$
		\mathcal{O}\big(\frac{\log\left(1/\varepsilon\right)
		}{\log(\log 1/\varepsilon) }\big)$} nonzero weights to
	approximate it with error $\varepsilon$.  More results are
	given in Theorem \ref{thm:1d_epsilon},
	\ref{thm:MdSobolev}, \ref{thm:MdsparseS2error}.
	
	\item The functions represented by RePU($s\ge 2$) networks are smooth
	functions, so they naturally fit in the places where
	derivatives are involved in the loss function.
	
	\item Compared to other high-order differentiable activation
	functions, such as logistic, $\tanh$, softplus, sinc etc.,
	RePUs are more efficient in terms of number of arithmetic
	operations needed to evaluate, especially the rectified
	quadratic unit.
\end{itemize}
Based on the facts above, we advocate the use of deep RePU
networks in places where the functions to be approximated
are smooth.

The remaining part of this paper is organized as follows. In
Section 2, we first show how to approximate univariate
smooth functions using RePU networks by converting best
polynomial approximations into RePU networks. Then we use a
similar approach to analyze the ReQU network approximation
for multivariate functions in weighted Sobolev space in
Section 3.  After that, we show how high-dimensional
functions with sparse polynomial approximations can be well
approximated by ReQU networks in Section 4.  Some
preliminary numerical results are given in Section 5. We end
the paper by a short summary in Section 6.

\section{Approximation of univariate functions by deep RePU networks}

We first introduce some
notations related to neural networks. Denote by $\bbN$ the set of all positive integers, 
$\bbN_0 := \{ 0 \} \cup \bbN$. Given $d,L\in\bbN$, we denote a neural
network $\Phi$ with input of dimension $d$, number of layer
$L$, by a matrix-vector sequence
\begin{equation} \label{eq:NNpara}
\Phi =  \big((A_{1},b_{1}), \cdots, (A_{L},b_{L}) \big),
\end{equation}
where $N_{0}=d$, $N_{1},\cdots,N_{L}\in \bbN$,
$A_{k}$ are $N_{k}\times N_{k-1}$ matrices, and
$b_{k}\in \bbR^{N_{k}}$.
If $\Phi$ is a neural network, and
$\rho:\bbR \to \bbR$ is an arbitrary activation function, then define
\begin{equation}\label{eq:NN1}
R_{\rho}(\Phi): \bbR^{d} \rightarrow \bbR^{N_{L}},
\qquad 
R_{\rho}(\Phi)(\bx)  =  \bx_{L},
\end{equation}
where $R_{\rho}(\Phi)(\bx)$ is given as 
\begin{equation}\label{eq:NN2}
\begin{cases}
\bx_{0} := \bx, & \\
\bx_{k} := \rho(A_{k}\bx_{k-1} + b_{k}), & k=1, 2,\ldots, L-1,\\
\bx_{L} := A_{L} \bx_{L-1} + b_{L}, &
\end{cases}
\end{equation}
and
\begin{equation*}
	\rho(\by) := \left( \rho(y^1), \cdots, \rho(y^m) \right),
	\quad \forall\ \by = (y^1,\cdots,y^m) \in \bbR^{m}.
\end{equation*}
We use three quantities to measure the complexity of the
neural network: number of hidden layers, number of
nodes (i.e. activation units), and number of nonzero weights,
which are $L-1$, $\sum_{k=1}^{L-1} N_{k}$ and number of
non-zeros in $\{ (A_{k}, b_k), k=1, \ldots, L \}$,
respectively, for the neural network defined in
\eqref{eq:NNpara}. For convenience, we denote by $\# A$ the
number of nonzero components in $A$ for a given matrix or
vector $A$. For the neural network $\Phi$ defined in
\eqref{eq:NNpara}, we also denote its number of nonzero
weights as $\#\Phi:=\sum_{k=1}^L(\#A_k+ \# b_k)$.

In this paper we study the approximation property of smooth
functions by deep neural networks with RePUs as activation
units.  It seems that $\sigma_s$ networks were first used in
the classic works by Mhaskar and his coworkers (see
e.g. \cite{mhaskar_approximation_1993},
\cite{chui_neural_1994a}) to obtain high-order convergence
of neural network approximation.  $\sigma_s$ is also a
special case of piece-wise polynomial activation function,
which has been studied in \cite{
petrushev_approximation_1998} for shallow network
approximation. We also note that $\sigma_3$ has been used in
a deep Ritz method proposed recently to solve PDEs using
variational form~\cite{e_deep_2018}.

The construction of RePU networks adopted by Mhaskar
bases on the fact that a polynomial of degree $n$ in $d$
dimension can be represented by a linear combination of
$\binom{n+d}{d}$ number of monomials of the form
$\big(Ax+b\big)^n$, with each one using different affine
transform. To represent a polynomial of degree $n$ using
$\sigma_s$ neural network, they first compose $\sigma_s(x)$
for $k=\lceil \log_s n \rceil$ times, which result in
$\sigma_{s^k}(x)$. Then a neural network with one-layer
$\sigma_{s^k}(x)$ units of amount $\binom{n+d}{d}$ is capable to
accurately represent any polynomial of degree $n$. This kind
of construction give an optimal linear approximation result
for neural network using high order (the order is $s^k$)
sigmoidal activation functions. However, if regard the
constructed neural network as a $\sigma_s$ neural network,
it has $k$ hidden layers.  The corresponding linear
approximation bound is quasi-optimal due to this factor
$k$. Moreover, to find the corresponding network
coefficients to represent a given polynomial, one needs to
solve a Vandermonde-like matrix, whose condition number is
known grows geometrically (see
e.g. \cite{gautschi_optimally_2011}).  In this paper, we
propose a different approach which does not involve
any Vandermonde matrix of large size.

\subsection{Approximation by deep ReQU networks}\label{subsec:ReQU}

Our analyses relies upon the fact: $x$, $x^2$, $\ldots$, $x^s$, and $xy$ all can 
be realized by $\sigma_s$ neural networks with a few number
of coefficients.  We first give the result for $s=2$ case.
\begin{lemma}\label{lem:s2basic}
	For any $x, y\in\bbR$, the following identities hold:
	\begin{align}
		x^2 & = \beta^{T}_{2} \sigma_{2}(\omega_{2}  x), \label{eq:s2x2} \\
		x & = \beta_{1}^T \sigma_2(\omega_1 x +\gamma_1),\label{eq:s2x}\\
		xy & = \beta_{1}^{T} \sigma_{2}(\omega_{1} x +\gamma_{1}y), \label{eq:s2xy}
	\end{align}	
	where
	\begin{equation} \label{eq:const_vecs}
	\beta_2=[1,1]^T, \
	\omega_2=[1,-1]^T, \ 
	\beta_1=\frac14[1,1,-1,-1]^T, \
	\omega_1=[1,-1,1,-1]^T, \
	\gamma_1=[1,-1,-1,1]^T.
	\end{equation}
	If both $x$ and $y$ are non-negative, the formula for
	$x^2$ and $xy$ can be simplified to the following form
	\begin{align}
		x^2 &= \sigma_2(x), \label{eq:s2x2pp}\\
		x y &= \beta_3^T \sigma_2(\omega_3 x +\gamma_{2}y), \label{eq:s2xypp}
	\end{align}
	where 
	\begin{equation}\label{eq:const_vecs_pp}
	\beta_{3}=\frac14[1, -1,-1]^T,\quad
	\omega_{3}=[1,1,-1]^{T},\quad
	\gamma_{2}=[1,-1,1]^{T}.
	\end{equation}
\end{lemma}
\begin{proof}
	All the identities can be obtained by straightforward calculations.
\end{proof}

Note that the realizations given in Lemma \ref{lem:s2basic}
are not unique. For example, to realize
$id_{\mathbb{R}}(x)=x$, we may use
\begin{equation*}
	x = (x+1/2)^2- x^2 - 1/4 = \beta_2^T \sigma_2
	(\omega_2(x+1/2)) - \beta_2^T \sigma_2(\omega_2 x) -1/4,
\end{equation*}
for general $x\in \bbR$, and use
\begin{equation*}
	x = (x+1/2)^2- x^2 - 1/4 = \sigma_2 (x+1/2) - \sigma_2(x)
	-1/4,
\end{equation*}
for non-negative $x$. To have a neat presentation, we will
use \eqref{eq:s2x2}-\eqref{eq:const_vecs_pp} throughout this
paper even though simpler realizations may exist for some
special cases.  We notice that the realization of the
identity map $id_{\mathbb{R}}(x)$ given in (\ref{eq:s2x}) is
a special case of $(\ref{eq:s2xy})$ with $y=1$. Furthermore,
the constant function $1$ can be represented by a trivial
network with $L=1$ and $A_1=0, b_1=1$ .

\begin{remark}\label{rem:x2ReLU}
	Notice that in
	\cite{yarotsky_error_2017,petersen_optimal_2018,montanelli_deep_2017},
	all the analyses rely on the fact that $x^2$ can be
	approximated to an error tolerance $\varepsilon$ by a
	deep ReLU networks of complexity $\mathcal{O}(\log
	\frac{1}{\varepsilon})$.  In our approach, by replacing
	ReLU with ReQU, $x^2$ is represented with no error using
	a ReQU network with only one hidden layer and 2 hidden
	neurons. This simple replacement greatly simplifies the
	proofs of some existing deep neural network
	approximation bounds, improves the approximation rate and
	meanwhile reduces the network complexity.
\end{remark}

\subsubsection{Optimal realizations of polynomials by deep ReQU networks with no error}

The basic property of $\sigma_2$ given in Lemma
\ref{lem:s2basic} can be used to construct {deep} neural
network representations of monomials and polynomials. We
first show that the monomial $x^n,n>2$ can be represented
exactly by deep ReQU networks of finite size but not shallow
ReQU networks.

\begin{theorem}\label{thm:s2xn1d}
	A) The monomial $x^n, n\in \bbN$ defined on $\bbR$ can be
	represented exactly by a $\sigma_2$ network.  The number
	of network layers, number of hidden nodes and number of nonzero weights
	required to realize $x^n$ are at most $\lfloor\log_2
	n\rfloor+2$, $5 \lfloor\log_2 n\rfloor+5$ and $25
	\lfloor\log_2 n\rfloor + 14$, respectively. Here $\lfloor
	x \rfloor$ represents the largest integer not exceeding
	$x$ for $x\in\bbR$.
	
	B) For any $n>2$, $x^n$ can not be represented exactly by
	any ReQU network with less than $\lceil \log_2 n \rceil$
	hidden layers.
\end{theorem}
\begin{proof}
	1) We first prove part B. For a one-layer ReQU network
	with $N$ activation units, one input and one output, the
	function represented by the network can be written as
	\begin{align*}
		f_N(x) = \sum_{k=1}^N c_k \sigma_2( a_k x + b_k) + d,
	\end{align*}
	where $d$ and $a_k, b_k, c_k$, $k=1,\ldots, N$ are the
	parameters of the network. Obviously, $f_N$ is a piecewise
	polynomial with at most $N+1$ pieces in the intervals
	divided by distinct points of $x_k = -b_k/a_k$, $k=1,
	\ldots, N$(suppose the points are in ascending order). In
	each piece, $f_N$ is a polynomials of degree 2. Since a
	polynomial of degree at most 2 composed with another
	polynomial of degree at most 2 produces a polynomial of
	degree at most 4, so a ReQU network with two hidden layers
	can only represent piecewise polynomials of degree at most
	4. By induction, a ReQU network with $m$ hidden layers can
	only represent piecewise polynomials of degree at most
	$2^m$. So, with $m< \lceil \log_2 n \rceil$, a ReQU
	network with $m$ hidden layers can't exactly represent
	$x^n$.
	
	2) Now we give a constructive proof for part A.  We first express $n$ in binary
	system as follows:
	\begin{equation*}
		n = a_m\cdot 2^m + a_{m-1}\cdot 2^{m-1} + \cdots + a_1\cdot 2 + a_0,
	\end{equation*}
	where $a_{j}\in\{0,1\}$ for $j=0,1,...,m-1$, $a_m=1$, and
	$m=\lfloor\log_2 n\rfloor$. Then
	\begin{align*}
		x^n 
		= x^{2^m}\cdot x^{\sum\limits_{j=0}^{m-1}a_j2^{j}}.
	\end{align*}
	Introducing intermediate variables 
	\begin{equation*}\label{eq:xidef1}
		\xi_k^{(1)}:=x^{2^k},\qquad
		\xi_k^{(2)}:=x^{\sum\limits_{j=0}^{k-1}a_j 2^j}, 
		\qquad\text{for}\ 
		1\le k\le m,
	\end{equation*} then 
	\begin{equation} \label{eq:xiXnrel}
	x^n=\xi_m^{(1)}\xi_{m}^{(2)}.
	\end{equation}
	We use the iteration scheme
	\begin{align}\label{eq:xirec}
		\begin{cases}
			\xi_1^{(1)}=x^2,&\\
			\xi_1^{(2)}=x^{a_0},&\\
		\end{cases}
		& 	\quad 	\text{and} 	\quad 
		\begin{cases}
			\xi_{k}^{(1)} = (\xi_{k-1}^{(1)})^2, & \\
			\xi_{k}^{(2)} = (\xi_{k-1}^{(1)})^{a_{k-1}}
			\xi_{k-1}^{(2)}, &\\
		\end{cases}
		\text{for}\ 2\le k\le m,
	\end{align}
	and \eqref{eq:xiXnrel} to realize $x^n$. The outline of
	the realization is demonstrated in Fig. \ref{fig:1}.  In
	each iteration step, we need to realize two basic
	operations: $(x)^2$ and $(x)^{a_k}y$, where $x, y$ stands
	for $\xi_{k-1}^{(1)}, \xi_{k-1}^{(2)}$ respectively. Note
	that $(x)^2$ can be realized by Eqs. \eqref{eq:s2x2}
	and \eqref{eq:s2x2pp} in Lemma $\ref{lem:s2basic}$. For
	the operation $(x)^{a_j}y$, since $a_j\in \{ 0, 1\}$, by
	\eqref{eq:s2xy}, we have
	\begin{equation}
	x^{a_j} y
	= \Big( \dfrac{1+(-1)^{a_j}}{2}  + \dfrac{1-(-1)^{a_j}}{2} x \Big) y
	= 
	\beta^{T}_{1} \sigma_{2} \left( \omega_{1} (c^{+}_{j}  
	+  c^{-}_{j}  x) + \gamma_{1} y \right),
	\end{equation}
	where $c^{\pm}_{j} := \frac{ 1 \pm (-1)^{a_{j} } }{2}$.
	
	Now we describe the procedure in detail. 
	For $n\ge 3$, we follow the idea given
	in Eq. \eqref{eq:xirec} and Fig. $\ref{fig:1}$. The
	function $x^n$ is realized in $m+1$ steps, which are
	discussed below.
	\begin{figure}[t!]
		\includegraphics[width=0.99\textwidth]{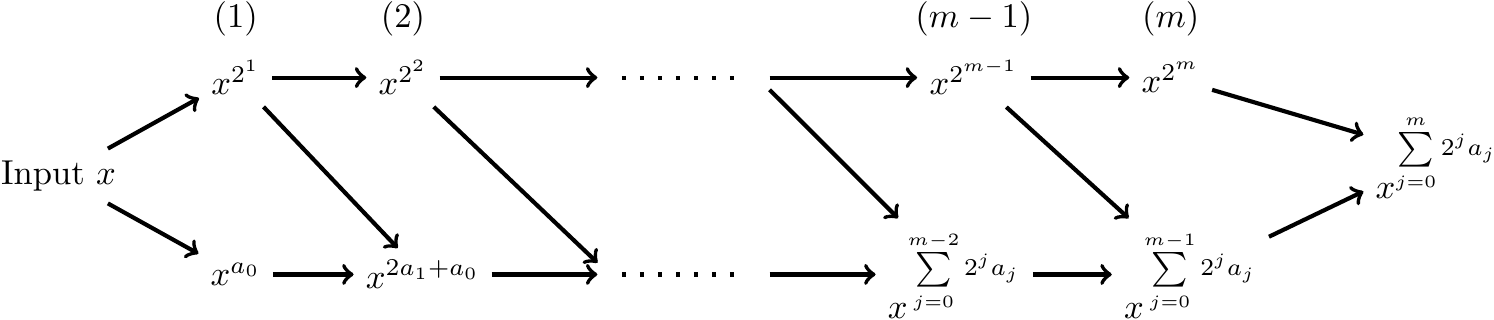}
		\caption{\label{fig:1} A schematic diagram for $\sigma_2$ network
			realization of $x^{n}$.  $(j)$ represents the $j$-th
			hidden layer for intermediate variables.}
	\end{figure}
	
	\begin{enumerate}
		\item[1)] In Step $1$, we calculate
		\begin{align}
			& \xi_1^{(1)} = x^2 = \beta_2^T \sigma_2(\omega_{2} x) \ge 0,\\
			& \xi_1^{(2)} = x^{a_0} = c^{+}_{0}+c^{-}_{0} x
			= c_{0}^{+} + c_{0}^{-} \beta^{T}_{1}
			\sigma_2 \left( \omega_{1}x + \gamma_{1} \right),
		\end{align}
		which implies the first layer output of the neural network is:
		\begin{align}\label{eq:out1}
			\bx_1 = \sigma_2(A_1x+b_1),
			\quad 
			\mbox{where}
			\quad
			{A}_{1}	= 
			\begin{bmatrix}
				\omega_{2}\\
				\omega_{1} 
			\end{bmatrix}_{6\times 1}, 
			\quad {b}_1 =
			\begin{bmatrix}
				\bm{0}\\
				\gamma_{1}
			\end{bmatrix}_{6\times 1},
		\end{align}
		and
		\begin{align}\label{eq:out1iv} 
			\begin{bmatrix}
				\xi_1^{(1)}\\
				\xi_1^{(2)}
			\end{bmatrix}
			= A_{20}\bx_1 + b_{20},
			\quad \text{where}\quad 
			A_{20} =
			\begin{bmatrix}
				\beta_2^T & \bm{0}\\
				\bm{0}& c_0^{-}\beta^{T}_{1}
			\end{bmatrix}_{2\times 6}, 
			\quad b_{20} =
			\begin{bmatrix}
				0\\
				c_0^{+}
			\end{bmatrix}_{2\times 1}.
		\end{align}
		Since $\# \omega_1 = 4$, $\# \omega_2 = 2$, $\# \gamma_1 =
		4$, it is easy to see that the number of nodes in the
		first hidden layer is $6$, and the number of non-zeros is:
		$\#{A}_{1} + \#{b}_{1}=10$.
		
		\item[2)] In Step $j$, $2\le j\le m$, we calculate
		\begin{align}
			\xi_{j}^{(1)}
			&  = (\xi_{j-1}^{(1)})^2
			= \sigma_2(\xi_{j-1}^{(1)})\ge 0,\\
			\xi_{j}^{(2)}
			&  =  (\xi_{j-1}^{(1)})^{a_{j-1}}  \xi_{j-1}^{(2)} 
			= ( c_{j-1}^{+} + c_{j-1}^{-} \xi_{j-1}^{(1)} ) \xi_{j-1}^{(2)} \nonumber\\
			& = \beta^{T}_{1}
			\sigma_2 \left(\omega_{1}( c_{j-1}^{+} + c_{j-1}^{-} \xi_{j-1}^{(1)} ) + \gamma_{1}\xi_{j-1}^{(2)}\right),
		\end{align}
		which suggest the $j$-th layer output of the neural
		network is:
		\begin{align}
			\bx_{j} 
			= \sigma_{2}
			\left( A_{j1}
			\begin{bmatrix}
				\xi_{j-1}^{(1)}\\
				\xi_{j-1}^{(2)}
			\end{bmatrix}
			+ b_{j1} \right),
			\quad \ 
			A_{j1}
			= \begin{bmatrix}
				1 & 0 \\
				c_{j-1}^{-}\omega_1 & \gamma_{1} \\
			\end{bmatrix}_{5\times 2}, 
			\ b_{j1} 
			=  \begin{bmatrix}
				0\\
				c_{j-1}^{+}\omega_1\\
			\end{bmatrix}_{5\times 1},  \nonumber
		\end{align}
		and
		\begin{align}\label{Aj0-bj0}
			\begin{bmatrix}
				\xi_{j}^{(1)}\\
				\xi_{j}^{(2)}
			\end{bmatrix}
			=  A_{j+1,0}\bx_{j}  +  b_{j+1,0},
			\quad \text{where} \
			A_{j+1,0} =
			\begin{bmatrix}
				1 &  \bm{0}\\
				0 &  \beta^{T}_{1} \\
			\end{bmatrix}_{2\times 5}, 
			\  b_{j+1,0} = \bm{0}.
		\end{align}
		We have
		\begin{align}
			{A}_j=A_{j1}A_{j0},\quad {b}_j=A_{j1}b_{j0}+b_{j1}, \quad
			j=2,\ldots, m.
		\end{align}
		By a direct calculation, we find that the number of nodes
		in Layer $j$ is $5\ (j=2,\ldots,m)$, and the number of
		non-zeros in Layer $j$, $j=3,\ldots,m$ is $\#{A}_{j} +
		\#{b}_{j} \leq 21+4=25$. For $j=2$, $\#{A}_{2} + \#{b}_{2}
		=26+4=30$.
		
		\item[3)] In Step $m+1$, we calculate
		\begin{equation}
		x^n
		= \xi_{m}^{(1)}\xi_{m}^{(2)} 
		= \beta^{T}_{1}\sigma_2
		\left(\omega_{1}\xi_{m}^{(1)}+\gamma_{1}\xi_{m}^{(2)}\right),
		\end{equation} 
		which implies
		\begin{align}
			\bx_{m+1}
			= \sigma_2\left(A_{m+1,1}
			\begin{bmatrix}
				\xi_{m}^{(1)}\\
				\xi_{m}^{(2)}
			\end{bmatrix}
			\right),
			\quad\mbox{where}\quad
			A_{m+1,1}
			= [\omega_{1}\ \gamma_{1}]_{4\times 2}.
		\end{align}
		So we get
		$\bm{x}_{m+1}=\sigma_{2}(A_{m+1}\bm{x}_{m}+b_{m+1})$, with
		\begin{equation}
		{A}_{m+1}=A_{m+1,1}A_{m+1,0},\quad
		{b}_{m+1}=\bm{0},
		\end{equation}
		and
		\begin{equation}
		\bx_{m+2} := x^n = \beta_1^T \bx_{m+1}.
		\end{equation}
		By a direct calculation, we get the number of nodes in
		Layer $m+1$ is $4$, the number of non-zero weights is
		$\#{A}_{m+1}=20$.
		
		For Layer $m+2$, which is the output layer of the overall
		network, ${A}_{m+2}=\beta^{T}_{1}$, and
		${b}_{m+2}=0$. There are no activation units and the number of
		nonzero weights is $\#A_{m+2} = 4$.
	\end{enumerate}
	
	The ReQU network we just built has $m+2$ layers. The total number of
	nodes is $6+5(m-1)+4=5m+5$. The total number of nonzero weights is
	$10+30+25(m-2)+20+4 = 25m+14$. Combining the cases $n=1,2$, we
	reach to the desired conclusion.
\end{proof}

Now we consider how to convert univariate polynomials into
$\sigma_2$ networks. If we directly apply Theorem
\ref{thm:s2xn1d} to each monomial term in a polynomial and
then combine them together, one would obtain a network of
depth $\mathcal{O}(\log_2 n)$ and size
$\mathcal{O}(n\log_2 n)$, which is not optimal.  We provide
here two algorithms to convert a polynomial into a ReQU
network of same scale, i.e. without the extra $\log_2 n$
factor. The first algorithm is a direct implementation of Horner's method (also known as Qin Jiushao's algorithm in China):
\begin{align}
	f(x) & = a_0  + a_1 x + a_2 x^2 + a_3 x^3 + \cdots + a_n x^n \nonumber\\
	& = a_0 + x\left(a_1 + x\left(a_2 + x\big(a_3+ \cdots +
	x(a_{n-1} + x a_n) \big) \right)\right). \label{eq:horner}
\end{align}
To describe the algorithm iteratively, we introduce the
following intermediate variables
\begin{equation*}
	y_k = \begin{cases}
		a_{n-1} + x a_n,  & k=n, \\
		a_{k-1} + x y_{k+1}, & k= n-1, n-2, \ldots, 1.
	\end{cases}
\end{equation*} 
Then we have $y_1 = f(x)$. By implementing of $y_k$ for
each $k$, using the realizations formula given in Lemma
\ref{lem:s2basic}, and stacking the implementations of $n$
steps up, we obtain a $\sigma_2$ neural network with
$\mathcal{O}(n)$ layers and  where each layer has a constant width
independent of $n$.

The second construction given in the following theorem can
achieve same representation power with same amount of
weights but much less layers.

\begin{theorem}\label{thm:s2Pn1d}
	If $f(x)$ is a polynomial of degree $n$ on $\bbR$ , then
	it can be represented exactly by a $\sigma_2$ neural
	network with $\lfloor\log_{2}n\rfloor+1$ hidden layers,
	and the numbers of nodes and nonzero weights are both of order
	$\mathcal{O}(n)$.  To be more precise, the number of nodes
	is bounded by $9n$, and number of nonzero weights is
	bounded by $61 n$.
\end{theorem}

\begin{proof}
	Assume $f(x)=\sum_{j=0}^n a_j x^j$, $a_n\neq 0$.
	We first use an example with $n=15$ to demonstrate the
	process of network construction as follows:
	\begin{align}
		f(x)
		& = a_{15}x^{15}+a_{14}x^{14}+\cdots+a_8x^8+ a_7x^7+a_6x^6+\cdots+a_1x+a_0 \nonumber \\
		& = 
		\underbrace{x^8}\limits_{\xi_{3,0}} \underbrace{
			\bigg\{\underbrace{x^4}\limits_{\xi_{2,0}}
			\underbrace{ \bigg[
				\underbrace{x^2}\limits_{\xi_{1,0}}
				\underbrace{(a_{15}x+a_{14})}\limits_{\xi_{1,8}}
				+\underbrace{(a_{13}x+a_{12})}\limits_{\xi_{1,7}}
				\bigg] }\limits_{\xi_{2,4}} 
			+ 
			\underbrace{
				\bigg[
				x^2\underbrace{(a_{11}x+a_{10})}\limits_{\xi_{1,6}}
				+\underbrace{(a_9x+a_8)}\limits_{\xi_{1,5}}\bigg] }\limits_{\xi_{2,3}} \bigg\} 
		}\limits_{\xi_{3,2}} \nonumber\\
		&\qquad + \underbrace{
			\bigg\{ x^4\underbrace{\bigg[x^2\underbrace{(a_7x+a_6)}\limits_{\xi_{1,4}}
				+ \underbrace{(a_5x+a_4)}\limits_{\xi_{1,3}}\bigg]}\limits_{\xi_{2,2}}
			+ \underbrace{\bigg[x^2\underbrace{(a_3x+a_2)}\limits_{\xi_{1,2}}
				+ \underbrace{(a_1x+a_0)}\limits_{\xi_{1,1}}\bigg]}\limits_{\xi_{2,1}} \bigg\}
		}\limits_{\xi_{3,1}}.  
	\end{align}
	Here $\xi_{1,j_{1}}, j_{1}=0, 1,2,\cdots,8$,
	$\xi_{2,j_2}, j_{2}=0, 1,2,\cdots,4$, and $\xi_{3,j_3}$,
	$j_3=0, 1,2$ are the intermediate variable output of Layer
	$1$, $2$, $3$, respectively. The final output is
	$f(x)=\xi_{3,0} \xi_{3,2} + \xi_{3,1}$.
	
	We first describe the construction for the case $n \ge 4$
	here.
	
	Denote $m=\lfloor\log_2 n\rfloor$. We first extend $f(x)$
	to include monomials up to degree $2^{m+1}-1$ by zero
	padding:
	\begin{equation}
	f(x) := \sum_{j=0}^{2^{m+1}-1}a_j x^j,
	\qquad \text{where} \quad
	a_j = 
	0,   \quad \text{for}\ n+1\le j\le 2^{m+1}-1.
	\end{equation}
	The process of building a $\sigma_2$ network to represent
	$f(x)$ is similar to the case $n=15$. We give details
	below.
	
	\begin{enumerate}
		\item[1)] The output of Layer $1$ intermediate variables
		are:
		\begin{align} 
			\xi_{1,j}  & = a_{2j-1}x+a_{2j-2}
			= a_{2j-1}\beta^{T}_{1} \sigma_2\left(\omega_{1}x+ \gamma_{1}\right) + a_{2j-2},
			\quad j=1, 2, ...,2^{m},  \label{eq:xi1j}\\
			\xi_{1,0} & = x^2
			=\beta_2^T\sigma_2(\omega_{2}x),
		\end{align}   
		which suggest
		\begin{align}
			\bx_{1} = \sigma_{2}
			\begin{pmatrix}
				\omega_{1}x + \gamma_{1} \\
				\omega_{2}x
			\end{pmatrix}
			= \sigma_{2}({A}_{1} x +{b}_{1}),
			\quad \text{where}\quad
			{A}_1 =
			\begin{bmatrix}
				\omega_{1}\\
				\omega_{2}\\
			\end{bmatrix}, \quad {b}_1 =
			\begin{bmatrix}
				\gamma_{1}\\
				\bm{0}
			\end{bmatrix}.
		\end{align}
		and
		\begin{align}\label{eq:s2PnL1iv}
			\bm{\xi}_1 
			= A_{2,0}\bx_1+b_{2,0},
			\quad \text{where}\quad
			A_{2,0} =
			\begin{bmatrix}
				a_{21}\beta^{T}_{1} & \bm{0}\\
				\bm{0}& \beta_2^T\\
			\end{bmatrix}, \quad b_{2,0} =
			\begin{bmatrix}
				a_{22}\\
				0\\
			\end{bmatrix},
		\end{align}
		with
		$\bm{\xi}_1=[\xi_{1,1}, \xi_{1,2}, \ldots, \xi_{1,2^m},
		\xi_{1,0}]^T$,
		$a_{21}=[a_1, a_3, \ldots, a_{2^{m+1}-1}]^T$, 
		$a_{22}=[a_0, a_2, \ldots, a_{2^{m+1}-2}]^T$.
		
		\item[2)] The output of Layer $2$ intermediate variables
		are:
		\begin{align}
			\xi_{2,j}  & = \xi_{1,0} \xi_{1,2j}+\xi_{1,2j-1} \nonumber\\
			& = \beta^{T}_{1}\sigma_2( \omega_{1}\xi_{1,2j} + \gamma_{1}\xi_{1,0})
			+ \beta^{T}_{1} \sigma_2\left(\omega_{1}\xi_{1,2j-1}+ \gamma_{1}\right),
			\quad
			j=1, 2, ...,2^{m-1},\\
			\xi_{2,0}  & = (\xi_{1,0})^2=\sigma_2(\xi_{1,0}),
		\end{align} 
		which imply
		\begin{align}\label{A21-b21}
			\bx_2 = \sigma_2 ( A_{21}
			\bm{\xi_1}  + b_{21} ),
			\quad  
			\bm{x_2},b_{21} \in \bbR^{(8\cdot 2^{m-1}+1)\times 1}, 
			\ A_{21}\in \bbR^{(8\cdot 2^{m-1}+1)\times(2^m+1)},
		\end{align}
		and most elements in $A_{21},b_{21}$ are zeros.  The
		nonzero elements are given below using a Matlab subscript
		style as:
		\begin{align}\label{eq:A21b21}
			A_{21}(8j-8:8j,[2j\!-\!1, 2j,2^m+1])=
			\begin{bmatrix}
				\omega_{1}&\bm{0} & \bm{0}\\
				\bm{0}     & \omega_{1} & \gamma_{1}\\
			\end{bmatrix}, \quad b_{21}(8j-8:8j)=
			\begin{bmatrix}
				\gamma_{1}\\
				\bm{0}
			\end{bmatrix},
		\end{align}
		for $j=1,2,\ldots,2^{m-1}$, and the last element of
		$A_{2,1}$ is $1$.  According to the result
		\eqref{eq:s2PnL1iv} of Layer $1$, we get
		\begin{align}
			\bx_2 =	\sigma_2 \left( A_2 \bx_1 + b_2 \right),\quad
			{A}_2=A_{2,1}A_{2,0},\quad {b}_2=A_{2,1}b_{2,0}+b_{2,1}.
		\end{align}
		We also have
		\begin{align}\label{eq:s2PnL2iv}
			\bm{\xi}_2 = A_{3,0}\bx_2,
			\quad \text{where}\quad
			A_{3,0} =
			\begin{bmatrix}
				I_{2^{m-1}}\otimes [\beta_1^T\ \beta_1^T]  & \bm{0}\\
				\bm{0}& 1\\
			\end{bmatrix},
		\end{align}
		Here
		$\bm{\xi}_2=[\xi_{2,1}, \xi_{2,2}, \ldots, \xi_{2,2^{m-1}},
		\xi_{2,0}]^T$, and $I_{2^{m-1}}$ is the identity matrix 
		in $\bbR^{2^{m-1}}$. $\otimes$ stands for Kronecker
		product.
		
		\item[3)] The output of Layer $k\ (3\le k \le m)$
		intermediate variables are:
		\begin{align}
			\xi_{k,j} & = \xi_{k-1,0} \xi_{k-1,2j} + \xi_{k-1,2j-1}
			\qquad
			(j=1, 2, ...,2^{m-k+1}) \nonumber\\
			& = \beta^{T}_{1}\sigma_2(\omega_{1}\xi_{k-1,2j} + \gamma_{1}\xi_{k-1,0})
			+  \beta^{T}_{1}\sigma_2\left(\omega_{1}\xi_{k-1,2j-1}+ \gamma_{1}\right), 
			\\
			\xi_{k,0} & = (\xi_{k-1,0})^2
			= \sigma_2(\xi_{k-1,0}).
		\end{align}
		Denote
		$\bm{\xi}_k=[\xi_{k,1}, \xi_{k,2}, \ldots,
		\xi_{k,2^{m-k+1}}, \xi_{k,0}]^T$. We have
		\begin{equation}\label{eq:Lk}
		\bm{\xi}_k = A_{k+1,0}\bx_{k}, 
		\quad
		\bx_{k} = \sigma_2 (A_{k1} \bm{\xi}_{k-1} + b_{k1}),
		\end{equation}
		where $A_{k1},b_{k1}$ have the same formula as
		$A_{21},b_{21}$ given in \eqref{eq:A21b21} except that the
		maximum value of $j$ is $2^{m-k+1}$ rather than $2^{m-1}$,
		and $A_{k+1,0}$ has the same formula as $A_{30}$ given in
		\eqref{eq:s2PnL2iv} with $\bm{1}_{2^{m-1} \times 1}$
		replaced by $\bm{1}_{2^{m-k+1} \times 1}$ and
		$\bm{1}_{n}=[1,\ldots,1]^{T}\in \mathbb{R}^{n\times 1}$.
		Combining \eqref{eq:Lk} and \eqref{eq:s2PnL2iv}, we get
		\begin{align}
			\bm{x}_k = \sigma_2( A_k \bx_{k-1} + b_k ),
			\quad \text{where}\ 
			{A}_{k}=A_{k1}A_{k0},\quad {b}_{k}=b_{k1}.
		\end{align}
		
		\item[4)] The output of Layer $m+1$ intermediate variables
		are:
		\begin{align}
			\xi_{m+1,1} = \xi_{m,0}\xi_{m,2}+\xi_{m,1}
			= \beta^{T}_{1}\sigma_2(\omega_{1}\xi_{m,2}+\gamma_{1}\xi_{m,0})
			+\beta^{T}_{1}\sigma_2\left(\omega_{1}\xi_{m,1}+ \gamma_{1}\right).
		\end{align}
		Written into the following form
		\begin{equation}\label{eq:Lm1}
		\bm{\xi}_{m+1} := [ \xi_{m+1,1}] = A_{m+2,0}\bx_{m+1}, 
		\quad
		\bx_{m+1} = \sigma_2 (A_{m+1,1} \bm{\xi}_{m} + b_{m+1,1}),
		\end{equation}
		we have
		\begin{align}
			A_{m+1,1}=
			\begin{bmatrix}
				\omega_{1} & \bm{0} & \bm{0}\\
				\bm{0} & \omega_{1} & \gamma_{1}\\
			\end{bmatrix}, \quad b_{m+1,1} =
			\begin{bmatrix}
				\gamma_{1}\\
				\bm{0}
			\end{bmatrix},
		\end{align}
		and
		\begin{equation}
		{A}_{m+2,0}=[\beta^{T}_{1}\   \beta^{T}_{1}], \quad {b}_{m+2,0}=0.
		\end{equation}
		The iteration formula for $\bx_{m+1}$ is $\bx_{m+1} =
		\sigma_2( A_{m+1} \bx_{m} + b_{m+1} )$, where
		\begin{align}
			{A}_{m+1} = A_{m+1,1}A_{m+1,0},  \quad 
			{b}_{m+1} = b_{m+1,1}.
		\end{align}
		\item[5)] Since $\bm{\xi}_{m+1}=f(x)$, the network ends at
		Layer $m+2$, with $\bx_{m+2}=\bm{\xi}_{m+1}$. So we get
		${A}_{m+2}=A_{m+2,0}$, and $b_{m+2}=0$ from Eq.
		\eqref{eq:Lm1}.
	\end{enumerate}
	For $n<4$, the procedure can be obtained by removing some
	sub-steps from the cases $n\ge 4$.  From the
	construction process, we see that the number of layers is $m+2$, the
	numbers of nodes in Layer 1 to Layer $m+1$ are
	$6$, $8\times 2^{m-k+1}+1(2\le k \le m)$ and 8 respectively,
	and the number of nonzero weights in $\bm{A}_j$,
	$\bm{b}_j(1\le j\le m+2)$ are not bigger than 10,
	$(40\times 2^{m-1}+2)+8\times 2^{m-1}$,
	$(68\times 2^{m-j+1}+1)+4\times2^{m-j+1}(3\le j\le m)$, 72,
	8 respectively. Summing up these numbers, we reach the
	desired bound.
\end{proof}

\begin{remark}
	Theorem \ref{thm:s2xn1d} says we can use a $\sigma_2$
	network of scale $\mathcal{O}(\log_2 n)$ to represent
	$x^n$ exactly.  Theorem \ref{thm:s2Pn1d} says that any
	polynomial of degree less than $n$ can be represented
	exactly by a $\sigma_2$ neural network with
	$\lfloor\log_{2}n\rfloor+1$ hidden layers, and no more
	than $\mathcal{O}(n)$ nonzero weights. Such results are
	not available for ReLU network and neural networks using
	other non-polynomial activation functions, such as
	$\operatorname{logistic}$, $\tanh$,
	$\operatorname{softplus}$, $\operatorname{sinc}$ etc.  We
	note that the constants in the two theorems may not be
	optimal, but the orders of number of layers and number of
	nonzero weights are optimal.	
\end{remark}

\subsubsection{Error bounds of approximating smooth functions by
	deep ReQU networks}

Now we analyze the error of approximating general smooth
functions using ReQU networks. We first introduce some
notations and give a brief review of some classical results
of polynomial approximation.

Let $\Omega \subseteq \bbR^d$ be the domain on which the
function to be approximated is defined. For the
1-dimensional case in this section, we focus on $\Omega =I
:= [-1,1]$. Similar discussions and results can be extended
to $\Omega=[0,\infty)$ and $(-\infty, \infty)$ as well. We
denote the set of polynomials with degree up to $N$
defined on $\Omega$ by ${P}_N(\Omega)$, or simply
${P}_N$. Let $J^{\alpha, \beta}_{n}(x)$ be the Jacobi
polynomial of degree $n$, $n\in \bbN_0$; the family of all
these polynomials forms a complete set of orthogonal
bases in the weighted $L^2_{\omega^{\alpha,
		\beta}}(I)$ space with respect to weight
$\omega^{\alpha, \beta}(x)=(1-x)^{\alpha}(1+x)^{\beta}$
for $\alpha, \beta>-1$. To describe functions with high
order regularity, we define the Jacobi-weighted Sobolev
space $B_{\alpha, \beta}^{m} (I)$ as (see
e.g. \cite{shen_spectral_2011}):
\begin{equation}\label{eq:wtSobolevSpace}
B_{\alpha, \beta}^{m} (I)
:= \left\{ u : \partial_{x}^{k}u \in L_{\omega^{\alpha+k, \beta+k} }^2 (I) ,
\quad 0 \leq k \leq m \right\} , \quad m \in \bbN,
\end{equation}
with norm
\begin{equation}\label{eq:wtSobolevSpaceNorm}
\| f \|_{B^{m}_{\alpha,\beta}}
:= \bigg(
\sum_{k=0}^m \big\| \partial_x^k u \big\|^2_{L^2_{\omega^{\alpha+k,\beta+k}}}
\bigg)^{1/2}.
\end{equation}
Define the $L^2_{\omega^{\alpha,\beta}}$-orthogonal
projection $\pi^{\alpha,\beta}_N$:
$L^2_{\omega^{\alpha,\beta}}(I) \rightarrow P_N$ by requiring
\begin{equation}
\left( \pi_N^{\alpha, \beta} u - u, v\right)_{\omega^{\alpha,\beta}} = 0,
\quad
\forall\, v\in P_N.
\end{equation}
A detailed error estimate on the projection error
$\pi_N^{\alpha,\beta}u - u$ is given in Theorem 3.35 of
\cite{shen_spectral_2011}, by which we have the following
theorem on the approximation error of ReQU networks.
\begin{theorem}\label{thm:s2approxN}
	Let $\alpha,\beta>-1$, $N\ge 1$. For any
	$u\in B^{m}_{\alpha, \beta}(I)$, there exist a ReQU
	network $\Phi^{u}_{N}$ with $\lfloor\log_{2}N\rfloor+1$
	hidden layers, $\mathcal{O}(N)$ nodes, 
	and $\mathcal{O}(N)$ nonzero weights, satisfying the
	following estimates.
	
	1) If $0\le l \le m \le N+1$, we have
	\begin{align}\label{eq:err1d_proj1}
		\left\|\partial^{l}_{x}\left(
		R_{\sigma_{2}}(\Phi^{u}_{N}) - u
		\right)\right\|_{\omega^{\alpha+l, \beta+l}}
		\leq 
		c  \sqrt{ \dfrac{(N-m+1)!}{(N-l+1)!} }
		(N+m)^{(l-m)/2}  \|\partial^{m}_{x}u\|_{\omega^{\alpha+m, \beta+m}}.
	\end{align}
	
	2) If {$m>N+1 \geq l$}, we have
	\begin{align}\label{eq:err1d_proj2}
		\left\| \partial^{l}_{x}\left(
		R_{\sigma_{2}}(\Phi^{u}_{N}) - u
		\right)\right\|_{\omega^{\alpha+l, \beta+l}}
		\le c (2\pi N)^{-1/4} \left(\dfrac{\sqrt{e/2}}{N}\right)^{N-l+1}
		\|\partial^{N+1}_{x}u\|_{\omega^{\alpha+N+1, \beta+N+1} }.
	\end{align}
	Here $c\approx 1$ for $N\gg 1$.
\end{theorem}

\begin{proof}
	For any given $u\in B^{m}_{\alpha, \beta}(I)$, the
	polynomial $f = \pi^{\alpha,\beta}_N u \in P_N$. The
	projection error $ \pi^{\alpha,\beta}_N u-u $ is estimated
	by Theorem 3.35 in \cite{shen_spectral_2011}, which is
	\eqref{eq:err1d_proj1} and \eqref{eq:err1d_proj2}
	with $R_{\sigma_{2}}(\Phi^{u}_{N})$ replaced by $
	\pi^{\alpha,\beta}_N u$.  By Theorem \ref{thm:s2Pn1d}, $f$
	can be represented exactly by a ReQU network
	$\Phi^{u}_{N}$ with $\lfloor\log_{2}N\rfloor+1$ hidden
	layers, $\mathcal{O}(N)$ nodes, and $\mathcal{O}(N)$
	nonzero weights, i.e.  $R_{\sigma_{2}}(\Phi^{u}_{N}) =
	\pi^{\alpha,\beta}_N u$. We thus obtain estimation
	\eqref{eq:err1d_proj1} and \eqref{eq:err1d_proj2}.
\end{proof}

%\newpage
\begin{remark}
	In \eqref{eq:err1d_proj1} and \eqref{eq:err1d_proj2}, we
	allow the error measured in high-order derivatives,
	i.e. $l\ge 3$, because $R_{\sigma_2}(\Phi^u_N)$ is an
	exact realization of a polynomial, which is infinitely
	differentiable. In practice, if $\Phi^u_N$ is a trained
	network with numerical error, we can not measure the error
	with derivatives order $\ge 3$, since $\partial_x^3
	\sigma_2(x)$ is not in $L^2$ space.
\end{remark}

Based on Theorem \ref{thm:s2approxN}, we can analyze the
network complexity of $\varepsilon$-approximation of a given
function with certain smoothness. For simplicity, we only
consider the case with $l=0$.  The result is given in the
following theorem.
\begin{theorem} \label{thm:1d_epsilon} For any given
	function $f(x)\in B^{m}_{\alpha,\beta}(I)$ with norm less
	than $1$, where $m$ is either a fixed positive integer or
	infinity, and for $\varepsilon\in(0,1)$ small enough, there exists a ReQU network $\Phi^f_\varepsilon$
	with number of layers $L$, number of nonzero weights $N$
	satisfying
	\begin{itemize}
		\item if $m$ is a fixed positive integer, then $L =
		\mathcal{O}\left(\frac1m\log_2\frac{1}{\varepsilon}\right)$,
		and $N =
		\mathcal{O}\big({\varepsilon}^{-\frac{1}{m}}\big)$;
		\item if $m=\infty$, i.e. $f(x)\in B^{\infty}_{\alpha,\beta}(I)$, then
		$L=
		\mathcal{O}\left(\log_2\left(\log\frac{1}{\varepsilon}\right)\right)$,
		and
		$N=
		\mathcal{O}\left(\frac{\log(1/\varepsilon)}{\log_{2}(\log(1/\varepsilon))} \right)$,
	\end{itemize}
	that approximates $f$ within an error tolerance
	$\varepsilon$, i.e.
	\begin{align}
		\|R_{\sigma_{2}}(\Phi^{f}_{\veps}) - f\|_{\omega^{\alpha,\beta}(I) } \le \veps.
	\end{align} 
\end{theorem}

\begin{proof}
	For a fixed $m$, or $N\gg m$, we obtain from
	\eqref{eq:err1d_proj1} that
	\begin{equation}
	\| R_{\sigma_{2}}(\Phi^{u}_{N}) - u \|_{\omega^{\alpha,\beta}(I) }
	\le c N^{-m} \| \partial_x^m u\|_{\omega^{\alpha+m,\beta+m}}. 
	\end{equation}
	By above estimate, we obtain that to achieve an error
	tolerance $\varepsilon$ to approximate a function with
	$B^{m}_{\alpha,\beta}(I)$ norm less than $1$, it
	suffices to take
	$N=\left(\frac{c}{\varepsilon}\right)^\frac{1}{m}$. For
	fixed $m$, we have
	$N=\mathcal{O}\big({\varepsilon}^{-\frac{1}{m}}\big)$,
	the depth of the corresponding ReQU network is
	$L=\mathcal{O}\left(\frac{1}{m}\log_2\frac{1}{\varepsilon}\right)$.
	
	For $f\in B^{\infty}_{\alpha,\beta}$, by taking
	$m=\infty$ in Theorem \ref{thm:s2approxN}, we have
	\begin{equation}
	\| R_{\sigma_{2}}(\Phi^{u}_{N}) - u \|_{\omega^{\alpha,\beta}(I) }
	\le c (2\pi N)^{-\frac14}  \left(\dfrac{\sqrt{e/2}}{N}\right)^{N+1}  \| u\|_{B_{\alpha, \beta}^{\infty}}
	\le c' e^{-\gamma N} \| u\|_{B_{\alpha, \beta}^{\infty}},
	\end{equation}
	where $c'$ is a general constant, and $\gamma\approx
	\mathcal{O}(\log N)$ can be larger than any fixed
	positive number for sufficient large $N$. To approximate
	a function with $B^{\infty}_{\alpha,\beta}(I)$ norm less
	than $1$ with error $\varepsilon= c' e^{-\gamma N}$, it
	suffices to take
	$N=\frac{1}{\gamma}\log\left(\frac{c'}{\varepsilon}\right)$,
	which means
	$N=\mathcal{O}\left(\frac{\log(1/\varepsilon)}{\log_{2}(\log(1/\varepsilon))}
	\right)$. The depth of the corresponding ReQU network is
	$L=\mathcal{O}\left(\log_2\left(\log\frac{1}{\varepsilon}\right)\right)$. Here
	$\veps$ is assumed to be small enough such that
	$\log_2\big(\log\frac{c'}{\varepsilon}\big)$ is no less
	than 1.
	
\end{proof}

\subsection{Approximation by deep networks using general rectified power units}

The results of approximation monomials, polynomials and
general smooth functions by ReQU networks discussed in
Subsection \ref{subsec:ReQU} can be extended to general RePU
networks.

To keep the paper short, we only present the results on
approximating monomials with RePU in Theorem
\ref{thm:spbasic}.  The other results similar to ReQU
networks can be obtained but the details are quite lengthy,
we report them in a separate paper \cite{li_PowerNet_2019}.

\begin{theorem}\label{thm:spbasic}
	Regarding the problem of using $\sigma_{s}(x)\;(2\le s\in
	\bbN)$ neural networks to exactly represent monomial
	$x^{n}$, $n\in \bbN$, we have the following results:
	\begin{itemize}
		\item[(1)] If $s=n$, the monomial $x^n$ can be realized
		exactly using a $\sigma_s$ networks having only 1 hidden
		layer with two nodes.
		
		\item[(2)] If $1\le n < s$, the monomial $x^n$ can be
		realized exactly using a $\sigma_s$ networks having only 1
		hidden layer with no more than $2s$ nodes.
		
		\item[(3)] If $n > s \ge 2$, the monomial $x^n$ can be
		realized exactly using a $\sigma_s$ networks having
		$\lfloor\log_s n\rfloor+2$ hidden layers with no more than
		$(6s+2) (\lfloor\log_s n\rfloor+2)$ nodes, no more than
		$\mathcal{O}(25s^{2} \lfloor\log_s n\rfloor)$ nonzero
		weights.
	\end{itemize}
\end{theorem}

\begin{proof}
	
	(1) It is easy to check that $x^s$ has an exact $\sigma_s$
	realization given by
	\begin{equation}\label{eq:rhos}
	\rho_{s}(x) := \sigma_{s}(x) + (-1)^{s} \sigma_{s}(-x) = x^s.
	\end{equation}
	
	(2) For the case of $1\le n < s$, we consider the
	following linear combination
	\begin{align}\label{eq:xnrepre}
		a_0 + \sum_{k=1}^s a_{k} \rho_{s}(x+b_{k})
		& = a_0 + \sum^{s}_{k=1} a_{k}  \left( \sum^{s}_{j=0}C^s_j  b^{s-j}_{k}   x^{j} \right)
		= a_0+ \sum^{s}_{j=0}C^s_j \left( \sum^{s}_{k=1} a_{k} b^{s-j}_{k}  \right) x^{j},
	\end{align}
	where $a_0, a_k, b_k$, $k=1,\ldots, s$ are parameters to be
	determined. $C^s_j$ are binomial coefficients. The above
	expression is equal to $x^{n}$, provided that the
	parameters solve the following linear system:
	\begin{equation}\label{eq:RePU_Vandermonde}
	D_{s+1} \bm{a} :=
	\begin{bmatrix}
	1         &  1     & \cdots & 1  & 0 			 \\
	\vdots    & \vdots &        & \vdots       & \vdots 		 \\
	b_{1}^{s-n}  & b_{2}^{s-n}  & \cdots & b_{s}^{s-n} & 0  \\
	\vdots    & \vdots &        &        & \vdots 		 \\
	b_{1}^{s-1}  & b_{2}^{s-1}  & \cdots & b_{s}^{s-1} & 0  \\
	b_{1}^{s}  & b_{2}^{s}  & \cdots & b^{s}_{s} &  1 \\
	\end{bmatrix}
	\begin{bmatrix}
	a_{1} \\
	\vdots \\
	\cdot \\
	a_{s} \\
	a_{0}
	\end{bmatrix}
	= 
	\begin{bmatrix}
	0 \\
	\vdots \\
	\displaystyle{(C^{s}_{n})^{-1}} \\
	\vdots \\
	0 \\
	\end{bmatrix},
	\end{equation}
	where the top-left $s\times s$ submatrix of $D_{s+1}$ is
	a Vandermonde matrix, which is invertible as long as
	$b_k$, $k=1,\ldots, s$ are distinct. For simplicity, we
	choose $b_k$, $k=0,\ldots, s$ to be equidistant points,
	then \eqref{eq:RePU_Vandermonde} is uniquely
	solvable. Solving for $a_0, \ldots, a_s$ we obtain an
	exact representation of $x^n$ using \eqref{eq:xnrepre},
	which corresponds to a neural network having one hidden
	layer with no more than $2s$ $\sigma_s$ units.
	
	For example, when $s=2$, we may take $b_1=-1$, $b_1=1$.
	Solving Eq. \eqref{eq:RePU_Vandermonde} with $n=1$,
	we get $a_1=-\frac14$, $a_2=\frac14$, and $a_0=0$. Thus
	\begin{equation*}
		x = \frac{1}{4} \rho_2(x+1) - \frac14 \rho_2(x-1).
	\end{equation*}
	When $s=3$, take $b_1=-1$, $b_2=0$, $b_3=1$, we obtain
	\begin{align*}
		x & = \dfrac{1}{6} \big[\rho_{3}(x-1) - 2\rho_{3}(x) + \rho_{3}(x+1) \big], \\
		x^{2} & = \dfrac{1}{6} \big[\rho_{3}(x+1)  -  \rho_{3}(x-1) \big]  -  \dfrac{1}{3}.
	\end{align*}
	
	(3) Now, we consider the case $n > s \ge 2$, $n\in\bbN$. For
	any given numbers $y, z\in \mathbb{R}$, using the identity
	\begin{equation}
	y z = \frac14 \big[ (y+z)^2 - (y-z)^2 \big]
	\end{equation}
	and the fact that $(y+z)^2$, $(y-z)^2$ both can be
	realized exactly by a one layer $\sigma_s$ network with
	no more than $2s$ nodes, we conclude that the product
	$y z$ can be realized by one layer $\sigma_s$ network
	with no more than $4s$ nodes.  To realize $x^n$ by
	$\sigma_s$, we rewrite $n$ in the following form
	\begin{align}
		n = a_m\cdot s^{m}+ a_{m-1}\cdot s^{m-1} + \cdots + a_{1}\cdot s + a_{0}, \quad m=\lfloor \log_{s}m\rfloor,
	\end{align}
	where $a_{j}\in \{0, 1,\ldots,s-1\}$ for $j=0, 1,...,m-1$
	and $a_{m}=1$. So we have
	\begin{equation}
	x^n = (x^{s^m})^{a_m} (x^{s^{m-1}})^{a_{m-1}} \cdots (x^{s})^{a_1} (x)^{a_0}.
	\end{equation}
	Define $\xi_k = x^{s^k}$, $z_{k+1}=(\xi_k)^{a_k}$, $k=0, 1,
	\ldots, m$, and
	\begin{equation} \label{eq:xs_iter}
	y_2 = x^{a_0},\qquad  y_{k+2} = z_{k+1} y_{k+1} 
	\ \big(=(x^{s^k})^{a_k} y_{k+1} \big),\quad \text{for}\ k=1,\ldots, m,
	\end{equation}
	we have $y_{m+2} = x^n$. Eq. \eqref{eq:xs_iter} can
	be regarded as an iteration scheme, with iteration variables
	$\xi_k, y_k, z_k$, where the subscript $k$ stands for the
	iteration step. A schematic diagram for this iteration is
	given in Fig. \ref{fig:RePU_Xn}. Different to Theorem
	\ref{thm:s2xn1d}, for $s>2$, we need a deep $\sigma_s$
	neural network with $m+2$ hidden layers to realize $x^n,
	n>s$, due to the introduction of intermediate variables
	$z_k$.  In each layer, we need no more than $2+2s+4s$
	activation nodes to calculate $\xi_{k+1}=\rho_s(\xi_{k})$,
	$z_{k+1}=(\xi_{k})^{a_k}$, and $y_{k+1}=z_k y_k$. So in
	total we need no more than
	$(6s+2)(m+2)=\mathcal{O}{(6s\log_s n)}$ nodes.  A direct
	calculation shows that the number of nonzero weights in the
	network is no more than $\mathcal{O}( 25 s^{2} \log_{s}n )$.
	The theorem is proved.
\end{proof}

\begin{figure}[t!]
	\includegraphics[width=0.99\textwidth]{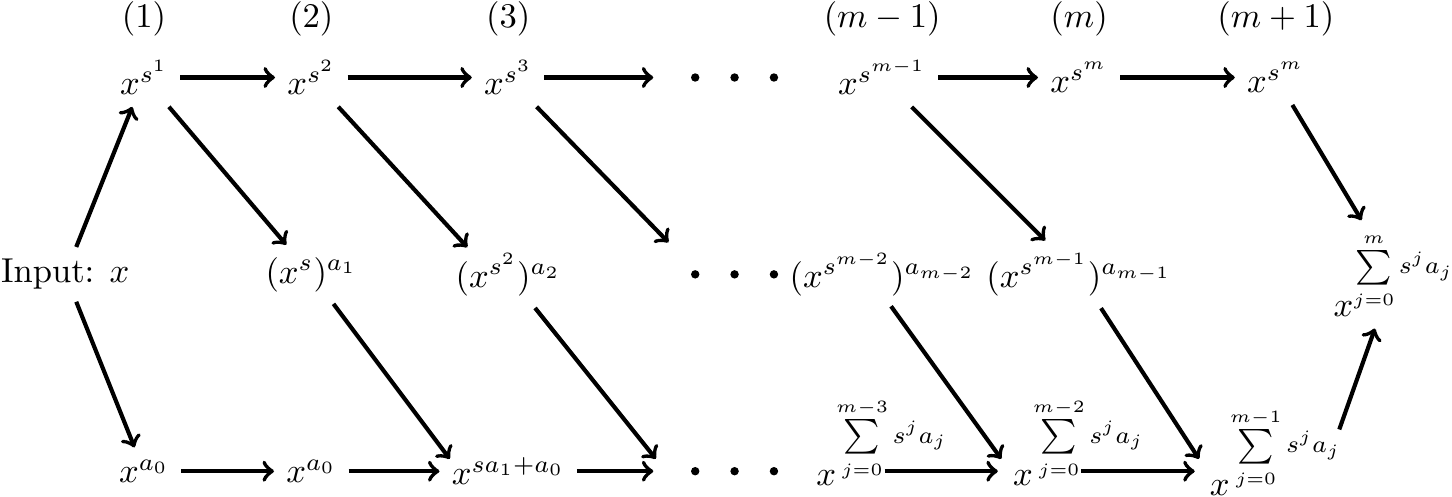}
	\caption{A schematic diagram for $\sigma_s$ network realization of $x^{n}$, $n>s$.
		$(j)$ represents the $j$-th hidden layer of
		intermediate variables.}
	\label{fig:RePU_Xn}
\end{figure}

\section{Approximation of multivariate functions}

In this section, we discuss how to approximate multivariate
smooth functions by ReQU networks.  Similar to the
univariate case, we first study the representation of
polynomials then discuss the approximation error of general
smooth functions.

\subsection{Deep ReQU network representations of multivariate polynomials}

\begin{theorem}\label{thm:mdpoly}
	If $f(x)$ is a multivariate polynomial with {\em total}
	degree $n$ on $\bbR^d$, then there exists a $\sigma_2$
	neural network having $d\lfloor\log_{2} n\rfloor+d$ hidden
	layers with no more than $\mathcal{O}(C^{n+d}_d)$
	activation functions and nonzero weights, that can
	represent $f$ with no error. We note that, here the
	constant behind the big $\mathcal{O}$ can be bounded independent of $d$.
\end{theorem}
\begin{proof}
	1) We first consider the 2-dimensional case. Suppose
	$f(x,y)=\sum_{i+j=0}^n a_{ij}x^iy^j$ and $n\ge 4$ (the results for $n\le 3$ are similar but
	easier, so skipped here).  To represent $f(x, y)$ exactly
	with a $\sigma_2$ neural network based on the results for
	the 1-dimensional case given in Theorem \ref{thm:s2Pn1d},
	we first rewrite $f(x,y)$ as
	\begin{align}
		f(x,y) 
		= \sum_{i=0}^n \bigg( \sum_{j=0}^{n-i} a_{ij} y^j \bigg) x^i
		=: \sum_{i=0}^n a^y_i x^i,
		\quad \text{where}\quad 
		a^y_{i} = \sum\limits_{j=0}^{n-i} a_{ij} y^j.
	\end{align}
	So to realize $f(x,y)$, we can first realize $a^y_i$,
	$i=0,\ldots, n-1$ using $n$ small $\sigma_2$ networks
	$\Phi_i$, $i=0,\ldots, n-1$, i.e.
	$R_{\sigma_2} (\Phi_i) (y) = a^y_i$ for given input $y$;
	then use a $\sigma_2$ network $\Phi_n$ to realize the
	1-dimensional polynomials
	$f(x,y) = \sum_{i=0}^n a^y_i x^i$.  There are two places that
	need some technical treatment, the details are given
	below.
	\begin{enumerate}
		
		\item[(1)] The network $\Phi_n$ takes $a^y_i$, 
		$i=0,\ldots,n$ and $x$ as input. So these quantities must
		be presented at the same layer of the overall neural
		network, because we do not want connections over
		non-adjacent layers.  By Theorem \ref{thm:s2Pn1d}, the
		largest depth of networks $\Phi_i$, $i=0,\ldots, n-1$ is
		$\lfloor\log_2 n\rfloor+2$, so we can lift $x$ to layer
		$\lfloor\log_2 n\rfloor+2$ using multiple
		$id_{\mathbb{R}}(\cdot)$ operations. Similarly, we also
		keep a record of input $y$ in each layer using multiple
		$id_{\mathbb{R}}(\cdot)$ operations, such that $\Phi_i$,
		$i=1,\ldots,n-1$ can start from appropriate layer and
		generate output exactly at layer $\lfloor\log_2
		n\rfloor+2$. The overall cost for recording $x, y$ in
		layers $1,\ldots, \lfloor\log_2 n\rfloor+2$ is
		$\mathcal{O}(\lfloor\log_2 n\rfloor+2)$, which is small
		comparing to the number of coefficients $C^{n+2}_2$.
		
		\item[(2)] While realizing $\sum_{i=0}^n a^y_i x^i$, the
		coefficients $a^y_i, i=0, \ldots n$ are network input
		instead of fixed parameters. So when applying the
		network construction given in Theorem \ref{thm:s2Pn1d},
		we need to modify the structure of the first layer of
		the network. More precisely, Eq. \eqref{eq:xi1j} in
		Theorem \ref{thm:s2Pn1d} should be changed to
		\begin{align}
			\xi^y_{1,j} & = a^y_{2j-1}x+a^y_{2j-2} \nonumber\\
			& = \beta^{T}_{1} \sigma_2\left(\omega_{1}x+ \gamma_{1}a^y_{2j-1}\right) 
			+ \beta^{T}_{1} \sigma_2\left(\omega_{1} a^y_{2j-2} + \gamma_{1} \right),
			\quad j=1,\ldots, 2^{m}.
		\end{align}
		So the number of nodes for the first layer changed from
		$6$ to $2+8\cdot 2^m$, the number of nonzero weights for
		the first layer changed from $10$ to $16\cdot 2^m+2$. So
		the number of hidden layers, number of nodes and nonzero
		weights of $\Phi_n$ can be bounded by $\lfloor\log_2
		n\rfloor+1$, $17n$, and $77n$ respectively.
	\end{enumerate}
	Assembling $\Phi_0, \ldots, \Phi_n$, the overall network
	to represent $f(x,y)$ has $2\lfloor\log_2 n\rfloor+3$
	layers with number of nodes no more than
	\begin{equation*}
		\sum_{j=0}^{n-1} 9 (n-j) + 17n + 8 (m+2) = 9\frac{n (n+1)}{2} + 17n + 8m + 16 =\mathcal{O}(C^{n+d}_{d}),
	\end{equation*}
	and number of weights no more than
	\begin{equation*}
		\sum_{j=0}^{n-1} 61 (n-j) + 77n + 16 (m+2)\times 2 + 12n = 61 \frac {n (n+1)}{2} + 89n + 32m + 64 =\mathcal{O}(C^{n+d}_{d}).
	\end{equation*}
	Thus, we proved that the theorem is true for the case
	$d=2$.
	
	2) The case $d>2$ can be proved by mathematical induction
	using the similar procedure as done for $d=2$ case. Note
	that we pad in some zeros in each direction in the
	iteration. Since after each dimension iteration, the
	number of degree of freedom are geometrically reduced, by
	a straightforward calculation, one can show that the
	constant behind the big $\mathcal{O}$ can be made
	independent of dimension $d$. An improved algorithm using
	less padding zeros is proposed in another
	paper~\cite{li_PowerNet_2019}.
	
\end{proof}

Using a similar approach as in Theorem \ref{thm:mdpoly}, one
can easily prove the following theorem.

\begin{theorem}\label{thm:mdpoly_tensor}
	For a polynomial $f_N$ in a tensor product space
	$Q_N^d(I_1\times\cdots \times I_d) :=
	P_N(I_1)\otimes\cdots\otimes P_N(I_d)$, there exists a
	$\sigma_2$ network having $d\lfloor\log_{2} N\rfloor+d$
	hidden layers with no more than $\mathcal{O}(N^d)$
	activation functions and nonzero weights, can represent
	$f_N$ with no error.
\end{theorem}

\subsection{Error bounds of approximating multivariate
 functions by ReQU networks}

Now we analyze the error of approximating general
multivariate smooth functions using ReQU networks.

For a vector $\bx=(x_1, \ldots, x_d) \in \bbR^d$, we define
$|\bx|_1:=|x_1| +\cdots + |x_d|$, $|\bx|_{\infty} :=
\max_{i=1}^d |x_i|$.  Define the high dimensional Jacobi
weight as $\omega^{\bm{\alpha},\bm{\beta}}(\bx) :=
\omega^{\alpha_1,
	\beta_1}(x_1)\cdots\omega^{\alpha_d, \beta_d}(x_d)$.
We define the multidimensional Jacobi-weighted Sobolev space
$B_{\alpha, \beta}^{m} (I^d)$ as \cite{shen_spectral_2011}:
\begin{equation*}\label{eq:wtSobolevSpaceMd}
	B_{\bm{\alpha}, \bm{\beta}}^{m} (I^d)
	:= \left\{ u\in L^{2}(I^{d}) \mid  \partial^{\bm{k}}_{\bx} u 
	:= \partial_{x_1}^{k_1}\cdots\partial_{x_d}^{k_d} u 
	\in L_{\omega^{\bm{\alpha}+\bm{k}, \bm{\beta}+\bm{k}} }^2 (I^d),
	\  \bm{k}\in \bbN_0^d, \ |\bm{k}|_1 \le m
	\right\} , \ m \in \bbN_0,
\end{equation*}
with norm and semi-norm
\begin{equation*}\label{eq:wtSobolevSpaceNormMd}
	\| u \|_{B^{m}_{\bm{\alpha},\bm{\beta}}} := \Big(
	\sum_{|\bm{k}|_1 \le m} 
	\left\| 
	\partial_{\bx}^{\bm{k}} u
	\right\|^2_{L^2_{\omega^{\bm{\alpha}+\bm{k}, \bm{\beta}+\bm{k}} }}
	\Big)^{1/2},
	\quad
	| u |_{B^{m}_{\bm{\alpha},\bm{\beta}}} := \Big(
	\sum_{\ |\bm{k}|_1 = m} \left\| 
	\partial_{\bx}^{\bm{k}} u
	\right\|^2_{L^2_{\omega^{\bm{\alpha}+\bm{k}, \bm{\beta}+\bm{k}} }}
	\Big)^{1/2}.
\end{equation*}

Define the
$L^2_{\omega^{\bm{\alpha},\bm{\beta}}}$-orthogonal
projection $\pi^{\bm{\alpha},\bm{\beta}}_N$:
$L^2_{\omega^{\bm{\alpha},\bm{\beta}}}(I^d) \rightarrow
Q_N^d(I^d)$ by the property
\begin{equation}
\left( \pi_N^{\bm{\alpha}, \bm{\beta}} u - u, v\right)_{\omega^{\bm{\alpha},\bm{\beta}}} = 0,
\quad
\forall\, v\in Q_N^d(I^d).
\end{equation}
Then for $u\in {B^{m}_{\bm{\alpha},\bm{\beta}}}(I^d)$, we
have the following error estimate (see Theorem 8.1 and
Remark 8.13 in \cite{shen_spectral_2011}):
\begin{equation}\label{eq:Mdpolyapprox}
\| \pi_N^{\bm{\alpha}, \bm{\beta}} u - u \|_{L^2_{\omega^{\bm{\alpha},\bm{\beta}}}(I^d)}
\le c N^{-m} 
| u|_{{B^{m}_{\bm{\alpha},\bm{\beta}}}},
\quad 1\le m \le N,
\end{equation}
where $c$ is an absolute constant. Combining
\eqref{eq:Mdpolyapprox} and Theorem \ref{thm:mdpoly_tensor},
we obtain the following upper bound for the
$\varepsilon$-approximation of functions in
${B^{m}_{\bm{\alpha},\bm{\beta}}}(I^d)$ space.

\begin{theorem}\label{thm:MdSobolev}
	For any $u\!\in\! {B^{m}_{\bm{\alpha},\bm{\beta}}}(I^d)$, with
	$|u|_{{B^{m}_{\bm{\alpha},\bm{\beta}}}(I^d)} \le 1,\ \bm{\alpha,\beta}\!\in\!(-1,\infty)^{d}$,  and any $\varepsilon\in(0,1)$ there
	exists a $\sigma_2$ neural network $\Phi_\varepsilon^u$
	having
	$\mathcal{O}\left(\frac{d}{m} \log_2 \frac{1}{\varepsilon}
	+ d\right)$ layers with no more than
	$\mathcal{O}\left(\varepsilon^{-d/m}\right)$ nodes and
	nonzero weights, that approximates $u$ with
	${L^2_{\omega^{\bm{\alpha},\bm{\beta}}}(I^d)}$-error less
	than $\varepsilon$, i.e.
	\begin{equation}
	\| R_{\sigma_2} (\Phi_\varepsilon^u) - u \|_{L^2_{\omega^{\bm{\alpha},\bm{\beta}}}(I^d)}
	\le \varepsilon.
	\end{equation}
\end{theorem}

\begin{remark}
	According to the classic nonlinear approximation theory
	by DeVore, Howard and
	Micchelli~\cite{devore_optimal_1989}, the results of
	Theorem \ref{thm:1d_epsilon} (first part) and Theorem
	\ref{thm:MdSobolev} are optimal in the case that the
	approximation depends on the function to be
	approximated continuously.
\end{remark}

\begin{remark}
	Note that results for approximating functions in
	weighted Sobolev space given in Theorem
	\ref{thm:MdSobolev} can be extended to $C^k$ if $k$ is
	sufficient large, similar to the second part of Theorem
	\ref{thm:1d_epsilon}.  Comparing this result with Theorem 1 in
	\cite{yarotsky_error_2017}, we see that the number of
	computational units and nonzero weights needed by a ReQU
	network to approximate a function $u\in
	{B^{m}_{\bm{\alpha},\bm{\beta}}}(I^d)$ for $m$
	sufficient large, with an error tolerance $\varepsilon$
	is less than that needed by a ReLU network. The ReLU
	network is $\log\frac{1}{\varepsilon}$ times larger than
	corresponding ReQU network. For low accuracy
	approximation, the factor
	$\mathcal{O}(\log\frac{1}{\varepsilon})$ is not very
	big, but for high accuracy approximations, this factor
	can be as large as several dozens, which could
	make a big difference in large scale computations.
	
	Note that, for functions with fixed lower order
	continuity, ReLU network can give good approximation using
	less number of layers, or use very deep ReLU networks to
	break the bounds given in Theorem \ref{thm:MdSobolev}. We
	refer interested readers to the recent works by
	Voigtlaender and Petersen
	\cite{voigtlaender_approximation_2019}, and Yarotsky
	\cite{yarotskyOptimalApproximationContinuous2018}.
	
\end{remark}

\section{High-dimensional functions with sparse
	polynomial\\ approximations}

In last section, we showed that for a $d$-dimensional
function with partial derivatives up to order $m$ in
$L^2(I^d)$ can be approximated within error $\varepsilon$ by
a ReQU neural network with complexity
$\mathcal{O}(\varepsilon^{-d/m})$. When $m$ is fixed or much
smaller than $d$, the network complexity has an exponential
dependence on $d$. However, in a lot of applications,
high-dimensional problems may have low intrinsic dimension
(see e.g. \cite{sloan_when_1998}\cite{wang_why_2005}).  One
particular example are high-dimensional tensor product
functions(or linear combinations of finite terms of tensor
product functions), which can be well approximated by a {\em
	hyperbolic cross} or {\em sparse grid} truncated series.

\subsection{A brief review of hyperbolic cross approximations and sparse grids}

Sparse grids were originally introduced by
S. A. Smolyak\cite{smolyak_quadrature_1963} to integrate or
interpolate high dimensional functions. Hyperbolic cross
approximation is a technique similar to sparse grids but
without the concept of grids. We introduce hyperbolic cross
approximation by considering a tensor product function:
$f(\bx) = f_1(x_1)\cdots f_d(x_d)$. Suppose that $f_1,
\ldots, f_d$ have similar regularity that can be well
approximated by using an orthonormal bases $\{ \phi_k,\; k=0,
1,\ldots. \}$; that is,
\begin{equation*}
	f_i(x) = \sum_{k=0}^\infty b_k^{(i)} \phi_k(x),\quad |b_k^{(i)}|\le c \bar{k}^{-r},  \quad i=1,\ldots, d,
\end{equation*}
where $c$ is a general constant, $r\ge 1$ is a constant
depending on the regularity of $f_i$, $\bar{k}:=\max\{1,
k\}$. So we have an expansion for $f$ as
\begin{equation*}
	f(\bx) 
	= \prod_{i=1}^d \bigg( \sum_{k=0}^\infty b_k^{(i)} \phi_k(x_i) \bigg)
	= \sum_{\bm{k}\in \bbN_0^d} b_{\bm{k}} \phi_{\bm{k}}(\bx),  
\end{equation*}
where
\begin{equation*}
	|b_{\bm{k}}| 
	= \big|b_{k_1}^{(1)}\cdots b_{k_d}^{(d)} \big|
	\le c^d (\bar{k}_1\cdots \bar{k}_d)^{-r},\quad 
	\phi_{\bm{k}}(\bm{x}) = \phi_{1}(x_{1})\cdots \phi_{d}(x_{d}).
\end{equation*}
Thus, to have a best approximation of $f(\bx)$ using finite
terms, one should take
\begin{equation} \label{eq:HyperbolicCrossDef}
f_N := \sum_{\bm{k}\in {\chi}_N^d} b_{\bm{k}} \phi_{\bm{k}}(\bx),
\end{equation}
where 
\begin{equation}
{\chi}_N^d := \left\{ 
\bm{k}=(k_1, \ldots, k_d)  \in \bbN_0^d
\mid 
\bar{k}_1\cdots \bar{k}_d \le N
\right\}
\end{equation}
is the hyperbolic cross index set. We call $f_N$ defined by
\eqref{eq:HyperbolicCrossDef} a hyperbolic cross
approximation of $f$.

For general functions defined on $I^d$, we choose
$\phi_{\bm{k}}$ to be multivariate Jacobi polynomials
$J_{\bm{n}}^{\bm{\alpha},\bm{\beta}}$, and define the
hyperbolic cross polynomial space as
\begin{equation} \label{eq:hyperbolic_poly}
X^d_N := \text{span}\{\, J_{\bm{n}}^{\bm{\alpha},\bm{\beta}},
\quad  
\bm{n}\in \chi^d_N
\,\}.
\end{equation} 
Note that the definition of $X_N^d$ doesn't depend
$\bm{\alpha}$ and $\bm{\beta}$.  $\{
J_{\bm{n}}^{\bm{\alpha},\bm{\beta}}\,\}$ is used to served
as a set of bases for $X_N^d$.  To study the error of
hyperbolic cross approximation, we define Jacobi-weighted
Korobov-type space
\begin{equation}\label{eq:Korobov}
\mathcal{K}^m_{\bm{\alpha},\bm{\beta}}(I^d)
:= \left\{\, 
u\in L^{2}_{\omega^{\bm{\alpha,\beta}}}(I^{d})
\ : \	 \partial^{\bm{k}}_{\bx} u  \in L^2_{\omega^{\bm{\alpha}+\bm{k},\bm{\beta}+\bm{k}}}(I^d),
\ 
0\le \ |\bm{k}|_\infty \le m
\,\right\},
\quad \text{for}\ m\in\bbN_0, 
\end{equation}  
with norm and semi-norm
\begin{equation}\label{eq:KorobovNormMd}
\| u \|_{\mathcal{K}^{m}_{\bm{\alpha},\bm{\beta}}} := \left(
\sum_{ |\bm{k}|_\infty \le m} \!\!
\left\| 
\partial_{\bx}^{\bm{k}} u
\right\|^2_{L^2_{\omega^{\bm{\alpha}+\bm{k}, \bm{\beta}+\bm{k}} }}
\right)^{\!\frac12},
\quad
| u |_{\mathcal{K}^{m}_{\bm{\alpha},\bm{\beta}}} := \left(
\sum_{\ |\bm{k}|_\infty = m} \!\! \left\| 
\partial_{\bx}^{\bm{k}} u
\right\|^2_{L^2_{\omega^{\bm{\alpha}+\bm{k}, \bm{\beta}+\bm{k}}}}
\right)^{\!\frac12}.
\end{equation}
For any given
$u\in \mathcal{K}^0_{\bm{\alpha},\bm{\beta}}
(=B^0_{\bm{\alpha},\bm{\beta}})$, the hyperbolic cross
approximation $\pi^{\bm{\alpha,\beta}}_{N,H}u\in X^{d}_{N}$ can be defined as a projection by requiring
\begin{equation}\label{eq:hyperbolic_proj}
(\pi_{N, H}^{\bm{\alpha},\bm{\beta}} u -u, v)_{\omega^{\bm{\alpha}, \bm{\beta}}} = 0, 
\quad
\forall\, v\in X_N^d.
\end{equation}
Then we have the following error estimate about the
hyperbolic cross approximation (see Theorem 2.2 in
\cite{shen_sparse_2010}):
\begin{equation}\label{eq:hyperbolic_error}
\| \partial^{\bm{l}}_{\bx} 
(\pi_{N, H}^{\bm{\alpha},\bm{\beta}} u - u)
\|_{\omega^{\bm{\alpha+l}, \bm{\beta+l}}}
\le D_1 N^{|\bm{l}|_\infty - m} 
|u|_{\mathcal{K}^{m}_{\bm{\alpha},\bm{\beta}}},
\quad 0\le \bm{l} \le m, \ m\ge 1,
\end{equation}
where $D_1$ is a constant independent of $N$.  It is known
that the cardinality of $\chi_N^d$ is of order
$\mathcal{O}(N(\log N)^{d-1})$ in
\cite{shen_sparse_2010}. The above error estimate says that
to approximate a function $u\in
\mathcal{K}^{m}_{\bm{\alpha},\bm{\beta}}$
with an error tolerance $\varepsilon$, one only needs a
space of Jacobi polynomials of dimension at most
$\mathcal{O}\left(\varepsilon^{-1/m}(\frac{1}{m} \log
\frac{1}{\varepsilon})^{d-1}\right)$,
the exponential dependence on $d$ is weakened (cp. Theorem
\ref{thm:MdSobolev}).  To remove the exponential term
$(\log\frac{1}{\varepsilon})^{d-1}$, one may consider a more
general sparse polynomial space\cite{shen_sparse_2010}:
\begin{equation}\label{eq:hc_poly_opt}
X^d_{N,\gamma} := \text{span}\big\{\, J_{\bm{n}}^{\bm{\alpha},\bm{\beta}},
\quad  
(\Pi_{i=1}^d\bar{n}_i) |\bm{n}|_{\infty}^{-\gamma} \le N^{1-\gamma}
\,\big\},
\quad 
-\infty \le \gamma < 1.
\end{equation}
In particular, $X^d_{N,0}=X^d_N$ is the hyperbolic cross
space defined in \eqref{eq:hyperbolic_poly}, and
$X^d_{N,-\infty} := \text{span}\big\{\,
J_{\bm{n}}^{\bm{\alpha},\bm{\beta}}, \ |\bm{n}|_{\infty} \le
N\,\big\}$ is the standard full grid.  For $0<\gamma<1$, it
is known that (see lemma 3 in \cite{griebel_sparse_2007}):
\begin{equation}\label{eq:card_opt_hc_space}
\text{Card}(X^d_{N,\gamma}) = C(\gamma,d) N, 
\quad 0< \gamma < 1,
\end{equation}
where $C(\gamma, d)$ is a constant that depends on $\gamma$
and $d$ but is independent of $N$. We call $X_{N,\gamma}^d,
0<\gamma<1$ optimized hyperbolic cross polynomial space. It
is proved by Shen and Wang that the
$L^2_{\omega^{\bm{\alpha},\bm{\beta}}}$-orthogonal
projection $\pi_{N,\gamma}^{\bm{\alpha},\bm{\beta}}$ from
Korobov space to $X_{N,\gamma}^d$ satisfies the following
estimate (see Theorem 2.3 in \cite{shen_sparse_2010}):
\begin{equation}\label{eq:opt_hc_approx}
\| \pi_{N,\gamma}^{\bm{\alpha},\bm{\beta}} u - u \|_{\omega^{\bm{\alpha},\bm{\beta}}}
\le D_2 N^{-m(1-\gamma(1-\frac{1}{d}))}
|u|_{\mathcal{K}^{m}_{\bm{\alpha},\bm{\beta}}},
\quad 0<\gamma < 1,
\end{equation}
where $D_2$ is a constant independent of $N$.  From
\eqref{eq:card_opt_hc_space} and \eqref{eq:opt_hc_approx},
we get that to approximate a function $u\in
\mathcal{K}^{m}_{\bm{\alpha},\bm{\beta}}$ with an error
tolerance $\varepsilon$, one only needs a space of Jacobi
polynomials of dimension at most
$\mathcal{O}\left(\varepsilon^{-1/m(1-\gamma(1-\frac{1}{d}))}\right)$.  We will later use this estimate to derive another upper bound of approximating functions in
${\mathcal{K}^{m}_{\bm{\alpha},\bm{\beta}}}$ using deep ReQU
networks.

In practice, the exact hyperbolic cross projection is not
easy to calculate. An alternate approach is the sparse
grid, which uses hierarchical interpolation schemes to build
a hyperbolic cross-like approximation of high dimensional
functions. To define sparse grids for $I^d$, we first define
the underlying 1-dimensional interpolations.  Given a series
of interpolation point sets $\mathcal{X}^{i}=\{ x^{
	i}_{1},\cdots,x^{i}_{m_{i}} \}\subseteq [-1,1]$, $m_i =
\text{Card}(\mathcal{X}^i)$, $i=1,2,\ldots $, with
$0<m_1<m_2 < \cdots$, the interpolation on $\mathcal{X}^i$
for $f\in C^0(I)$ is defined as
\begin{align}
	\mathcal{U}^{ i}(f)  = \sum^{m_{i}}\limits_{j=1}f(x^{ i}_{j})  \ell^{i}_{j}(x),
\end{align}
where $\ell^{i}_{j}(x) \in
P_{m_i-1}([-1,1])\ (j=1,2,\ldots,m_{i})$ are the Lagrange
interpolation polynomials for the interpolation points
$\mathcal{X}^{i}$.  The sparse grid interpolation for
high-dimension function $f\in C^0(I^d)$ is defined as
\cite{smolyak_quadrature_1963}:
\begin{align}\label{eq:spg_interp}
	\mathcal{A}(q,d) (f)
	= \sum_{d\leq |\bi|_1\leq q}
	\Big( \Delta^{ i_{1}}\otimes \cdots \otimes \Delta^{ i_{d}} \Big) (f),
	\quad q\ge d,
\end{align}
where $\Delta^i=\mathcal{U}^i - \mathcal{U}^{i-1}$,
$i\!\in\! \bbN$. For convenience, we define
$\mathcal{U}^0:=0$, $m_0=0$, $\mathcal{X}^0 = \emptyset$.
Formally, \eqref{eq:spg_interp} can be defined on any grids
$\{\, \mathcal{X}^i,\, i=1,2,\ldots, q-d+1 \,\}$. However,
to have a one-to-one transform between the values on
interpolation points and the coefficients of linearly
independent bases in the interpolation space, we need $\{\,
\mathcal{X}^i,\, i=1,2,\ldots, q-d+1 \,\}$ to be nested,
i.e.  $\mathcal{X}^1 \subset \mathcal{X}^2 \subset \cdots \subset
\mathcal{X}^{q-d+1}$. Fast transforms between physical
values and interpolation coefficients always exist for
sparse grid interpolations using nested grids
\cite{shen_efficient_2010, shen_efficient_2012}. Define sparse grid index set as
\begin{equation}
\mathcal{I}^q_d := \bigcup_{d\le |\bi|_1 \le q} 
\tilde{\mathcal{I}}^{i_1}\times \cdots \times 
\tilde{\mathcal{I}}^{i_d},
\quad
\text{where}\ 
\tilde{\mathcal{I}}^{k} := \mathcal{I}^{k}\setminus\mathcal{I}^{k-1},
\quad
\mathcal{I}^{k} = \{\, 1,2, \ldots, m_k \,\}.
\end{equation}
Then the set of the sparse grid interpolation points and the
corresponding interpolation space are given as
\begin{align}
	\mathcal{\mathcal{X}}^q_d
	= 
	\bigcup\limits_{d\leq |\bi|_1\leq q}
	\Big( (\mathcal{X}^{i_1}\setminus \mathcal{X}^{i_1-1})\times \cdots \times (\mathcal{X}^{i_1}\setminus \mathcal{X}^{i_1-1}) \Big),
	\quad q\ge d,
\end{align}
\begin{align}
	{V}^q_d = \text{span}\{\,
	\tilde{\phi}_{\bm{k}}(\bx), \, \bm{k}\in \mathcal{I}^q_d
	\,\} \quad q\ge d,
\end{align}
where $\tilde{\phi}_{\bm{k}}$ can be chosen as the
hierarchical interpolation basis defined in
\cite{shen_efficient_2010}, or the Lagrange-type
$d$-dimensional interpolation polynomial on points
$\mathcal{X}^q_d$, which takes value $1$ on $\bm{k}$-th
interpolation point and $0$ on the other points.

A commonly used 1-dimensional scheme is the
Chebyshev-Gauss-Lobatto scheme, which uses the extrema of
the Chebyshev polynomials as interpolation points:
\begin{align}
	x^{i}_{j}
	= -\cos\left(\dfrac{  (j-1)\pi }{ m_{i}-1  }  \right),
	\quad j=1, 2,\cdots,m_i.
\end{align}
In order to obtain nested sets of points, $m_i$ are chosen
as
\begin{align}
	m_{i} = \begin{cases}
		1, & i=1, \\
		2^{i-1} + 1,  & i>1,
	\end{cases}
\end{align}
with $x^1_1 :=0$.  Define
\begin{align}
	F^{k}_{d}
	:= \{\, f:[-1,1]^{d} \to \bbR \mid D^{\bm{\alpha}}f \in C([-1,1]^{d}),
	\ 
	\forall\ |\bm{\alpha}|_\infty \le k\;  \}.
\end{align}
Then for any function $f\in F^k_d$, with
$\|f\|_{F^k_d}:=\max_{|\bm{\alpha}|_{\infty}\leq
	k}\|\partial^{\alpha}f\|_{L^{\infty} }\le 1$, the
interpolation error on the above Chebyshev sparse grids are
bounded as Theorem 8 in \cite{barthelmann_high_2000}:
\begin{align}\label{eq:spg_interp_error}
	\| f - \mathcal{A}(q,d) f\|_{L^{\infty }} 
	\le c_{d,k}  2^{-kq} q^{2d-1}
	\le c_{d,k}  n^{-k}  (\log n)^{(k+2)(d-1)+1},
\end{align}
where
$n=\text{Card}(\mathcal{X}^q_d)=\text{Card}(\mathcal{I}^q_d)=\mathcal{O}(2^{q}q^{d-1})$
is the number of points in the sparse grids, and $c_{d,k}$
is a constant that depends on $d,k$ only.  Note that if a
different norm instead of the $L^\infty$ norm is used,
one can improve the result a little bit, but no results with
error bound smaller than $\mathcal{O}(n^{-k})$ is known.

\subsection{Error bounds of deep ReQU network approximation for
	multivariate\\ functions with sparse structures}

Now we discuss the ReQU network approximation of
high-dimensional smooth functions with sparse polynomial
expansions, which takes hyperbolic cross and sparse grid
polynomial expansions as examples.  We introduce the concept
of {\em downward closed} polynomial space first. A linear
polynomial space $P_C$ is said to be downward closed if it
satisfies the following: if $d$-dimensional polynomial $
p(\bx) \in P_{C}$, then $ \partial^{\bm{k}}_{\bx} p(\bx)
\in P_{C}$ for any $\bm{k}\in \bbN_0^d$, at the same time,
there exists a set of bases that is composed of monomials
only.  It is easy to verify that the hyperbolic cross
polynomial space $X^d_N$, the sparse grid polynomial
interpolation space $V^q_d$, and the optimized hyperbolic
cross space $X^d_{N,\gamma}$ are all downward closed.  For
a downward closed polynomial space, we have the following
ReQU network representation results.

\begin{theorem}
	\label{thm:Pcomplete_ReQU}
	Let $P_C$ be a downward closed linear space of
	$d$-dimensional polynomials with dimension $n$, then for
	any function $f\in P_C$, there exists a $\sigma_2$ neural
	network having no more than $\sum_{i=1}^d\lfloor\log_{2}
	N_i\rfloor+d$ hidden layers, no more than $\mathcal{O}(n)$
	activation functions and nonzero weights, can represent
	$f$ exactly. Here $N_i$ is the maximum polynomial degree
	with respect to the $i$-th coordinate.
\end{theorem}

\begin{proof}
	The proof is similar to Theorem \ref{thm:mdpoly}. First,
	$f$ can be written as a linear combination of monomials.
	\begin{equation}
	f(\bx) = \sum_{\bm{k}\in \chi_C} a_{\bm{k}} \bx^{\bm{k}},
	\end{equation}
	where $\chi_C$ is the index set of $P_C$ with cardinality
	$n$. Then we rearrange the summation as
	\begin{equation}
	f(\bx) = \sum_{k_d=0}^{N_d}  
	a_{k_d}^{x_1,\ldots, x_{k_{d\!-\!1}}}
	x_d^{k_d},
	\quad
	a_{k_d}^{x_1,\ldots,x_{k_{d\!-\!1}}}
	:=
	\sum_{(k_1,\ldots, k_{d\!-\!1})\in \chi_C^{k_d}} 
	a_{k_1,\ldots, k_{d\!-\!1}, k_d} x_1^{k_1} \cdots x_{d\!-\!1}^{k_{d\!-\!1}},
	\end{equation}
	where $\chi_C^{k_d}$ are $d-1$ dimensional downward closed
	index sets that depend on the index $k_d$.  If each
	$a_{k_d}^{x_1,\ldots,x_{k_{d-1}}}$, $k_d=0,1, \ldots, N_d$
	can be exactly represented by a $\sigma_2$ network with no
	more than $\sum_{i=1}^{d-1}\lfloor\log_{2}
	N_i\rfloor+(d-1)$ hidden layers, no more than
	$\mathcal{O}(\text{Card}(\chi_C^{k_d}))$ nodes and nonzero
	weights, then $f(x)$ can be exactly represented by a
	$\sigma_2$ neural network with no more than
	$\sum_{i=1}^d\lfloor\log_{2} N_i\rfloor+d$ hidden layers,
	no more than $\mathcal{O}(n)$ nodes and nonzero weights,
	since the operation $\sum_{k_d=0}^{N_d}
	a_{k_d}^{x_1,\ldots,x_{k_{d-1}}} x_d^{k_d}$ can be
	realized exactly by a $\sigma_2$ network with
	$\lfloor\log_{2} N_d\rfloor+ 1$ hidden layers and no more
	than $\mathcal{O}(N_d)$ nodes and nonzero weights.  So, by
	mathematical induction, we only need to prove that when
	$d=1$ the theorem is satisfied, which is true by Theorem
	\ref{thm:s2Pn1d}.
\end{proof}

\begin{remark}
	\label{rmk:spg_hyperbolic_ReQU}
	According to Theorem \ref{thm:Pcomplete_ReQU}, we have
	that:
	\begin{enumerate}
		\item[1)] For any $f\in X^d_N$, there exists a ReQU
		network with no more than $d\lfloor\log_{2} N\rfloor+d$
		hidden layers, no more than $\mathcal{O}(N(\log
		N)^{d-1})$ neurons and nonzero weights, that can
		represent $f$ with no error.
		
		\item[2)] For any $f\in X^d_{N,\gamma}$ with $0<\gamma<1$,
		there exists a ReQU network having no more than
		$d\lfloor\log_{2} N\rfloor+d$ hidden layers, no more
		than $\mathcal{O}(N)$ neurons and nonzero weights, that
		can represent $f$ with no error.
		
		\item[3)] For any $f\in V^q_d$, there exists a ReQU
		network having no more than $d(q-d+2)$ hidden layers, no
		more than $\mathcal{O}(2^{q}q^{d-1})$ neurons and
		nonzero weights, that can represent $f$ with no error.
	\end{enumerate} 
\end{remark}

Combining the results in Remarks
\ref{rmk:spg_hyperbolic_ReQU} with
\eqref{eq:hyperbolic_error}, \eqref{eq:opt_hc_approx} and
\eqref{eq:spg_interp_error}, we obtain the following
theorem.
\begin{theorem}
	\label{thm:MdsparseS2error}
	We have following results for ReQU network approximation
	of functions in
	$\mathcal{K}^{m}_{\bm{\alpha},\bm{\beta}}(I^d)$, $\bm{\alpha,\beta}\in(-1,\infty)^{d}$,
	$m\ge 1$ and $F^k_d(I^d)$, $k\ge 1$:
	\begin{enumerate}
		
		\item[1)] For any function $u\in
		\mathcal{K}^{m}_{\bm{\alpha},\bm{\beta}}(I^d)$, $m\ge
		1$ with
		$|u|_{\mathcal{K}^{m}_{\bm{\alpha},\bm{\beta}}} \le
		1/D_1$, any $\varepsilon > 0$, there exists a ReQU
		network $\Phi_\varepsilon^u$ with no more than $
		\frac{d}{m} \log_2 \frac{1}{\varepsilon} + d$ hidden
		layers, no more than
		$\mathcal{O}\big(\varepsilon^{-1/m}(\frac{1}{m} \log
		\frac{1}{\varepsilon})^{d-1}\big)$ nodes and nonzero
		weights, such that
		\begin{equation}
		\| R_{\sigma_2} (\Phi_\varepsilon^u) - u \|_{\omega^{\bm{\alpha},\bm{\beta}}}
		\le \varepsilon.
		\end{equation}
		
		\item[2)] For any function $u\in
			\mathcal{K}^{m}_{\bm{\alpha},\bm{\beta}}(I^d)$, $m\ge 1$
			with $|u|_{\mathcal{K}^{m}_{\bm{\alpha},\bm{\beta}}} \le
			1/D_2$, any $\varepsilon > 0$, $0<\gamma<1$, there exists a ReQU network
			$\Phi_\varepsilon^u$ with no more than $
			\frac{d}{m(1-\gamma(1-\frac{1}{d}))} \log_2
			\frac{1}{\varepsilon} + d$ hidden layers, no more than
			$\mathcal{O}\big(\varepsilon^{-1/[m(1-\gamma(1-\frac{1}{d}))]}\big)$
			nodes and nonzero weights, such that
			\begin{equation}
			\| R_{\sigma_2} (\Phi_\varepsilon^u) - u \|_{\omega^{\bm{\alpha},\bm{\beta}}}
			\le \varepsilon.
			\end{equation}	
		\item[3)] For any function $f\in F^k_d(I^d)$, $k\ge 1$ with
		$\|f\|_{F^k_d} \le 1$, any $\varepsilon >0$, there exists
		a ReQU network $\Psi_\varepsilon^f$ with no more than
		$\mathcal{O}\left( \frac{d}{k} \log_2
		\frac{1}{\varepsilon} + d\right)$ hidden layers, no more
		than
		$\mathcal{O}\big(\varepsilon^{-\frac{1+\delta}{k}}(\frac{1+\delta}{k}
		\log_2 \frac{1}{\varepsilon})^{d-1}\big)$ nodes and
		nonzero weights, such that
		\begin{equation}
		\| R_{\sigma_2} (\Psi_\varepsilon^f) - f \|_{L^\infty} \le \varepsilon,
		\end{equation}
		where $\delta>0$ can be taken very close to $0$ for small
		enough $\varepsilon$.
	\end{enumerate}
\end{theorem}

\begin{remark}
	Taking $m=2$ in Theorem \ref{thm:MdsparseS2error}, we obtain
	the following result: For any function $u\in
	\mathcal{K}^{2}_{\bm{\alpha},\bm{\beta}}(I^d)$, with
	$|u|_{\mathcal{K}^{2}_{\bm{\alpha},\bm{\beta}}} \le 1/D_1$,
	and $\varepsilon>0$ there exists a ReQU network
	$\Phi_\varepsilon^u$ with no more than $ \frac{d}{2} \log_2
	\frac{1}{\varepsilon} + d$ hidden layers, no more than
	$\mathcal{O}\big(\varepsilon^{-1/2}(\frac{1}{2} \log
	\frac{1}{\varepsilon})^{d-1}\big)$ nodes and nonzero
	weights, that approximates $u$ with a tolerance
	$\varepsilon$.  A result of using ReLU networks
	approximating similar functions is recently given by
	Montanelli and Du \cite{montanelli_deep_2017}. To
	approximate a function in
	$\mathcal{K}^{2}_{\bm{\alpha},\bm{\beta}}(I^d)$ with
	tolerance $\varepsilon$, they constructed a ReLU network
	with $\mathcal{O}(|\log_2 \varepsilon| \log_2 d )$ layers
	and $ \mathcal{O}(\varepsilon^{-\frac{1}{2}} |\log_2
	\varepsilon|^{\frac{3}{2}(d-1)+1} \log_2 d)$ nonzero
	weights. Comparing the two results, we find that, while the
	number of layers required by ReQU networks might be larger
	than ReLU networks, the overall complexity of the ReQU
	network is $|\log_2 \varepsilon |^d$ times smaller than that
	of ReLU network.
\end{remark}

\begin{remark}
	When one use optimized hyperbolic cross polynomial
	approximation for functions in
	$\mathcal{K}^{m}_{\bm{\alpha},\bm{\beta}}(I^d)$, with
	$|u|_{\mathcal{K}^{m}_{\bm{\alpha},\bm{\beta}}} \le 1/D_2$,
	the exponential growth on $d$ with a base related to
	$1/\varepsilon$ in the required ReQU network size is
	removed. Thus, in this case it seems that the curse of
	dimensionality does not exist any more. But we note that,
	the constant $D_2$ and the implicit constant hidden in the
	big $\mathcal{O}$ notation, still depend on $d$. In 
	practice, the error bound given by the second case may
	not be better than the first case.
\end{remark}

\section{Some preliminary numerical results}

In this section, we present some numerical results to verify
that the construction algorithms proposed are numerically
stable and efficient.  We first present the results of
representing univariate monomials in Table
\ref{tb:Monomial}. The maximum norm error in this table is
calculated by taking the maximum difference on 100 randomly
choose points in $[-1,1]$. The results show that the ReQU network we
constructed can achieve machine accuracy, which means our
approach is numerically stable.

\begin{table}[!htbp]
	\caption{Representation of monomials $x^n$. }
	\label{tb:Monomial}
	\centering
	\renewcommand{\arraystretch}{1.1}
	\begin{tabular}{rcccc}
		\hline
		Degree $n$ & $L$ &  $\#\mbox{weight}$ & $\#\mbox{node}$ & $L^{\infty}$-Error\\
		\hline
		3        & 3 &  38        & 10		 & 4.44e-16\\
		7        & 4 &  64        & 15		 & 2.22e-16\\
		15       & 5 &  89        & 20       & 9.99e-16\\
		31       & 6 &  114       & 25       & 7.77e-16\\
		63       & 7 &  139       & 30       & 6.11e-16\\
		127      & 8 &  164       & 35       & 2.22e-16\\	
		\hline 
	\end{tabular}
\end{table}	

Similar results for representing univariate polynomials are
given in Table \ref{tb:Poly-1D}.  Here, the coefficients of
the power series are generated randomly according to
standard normal distribution. These results also verify our
approach is stable and efficient.

\begin{table}[!htbp]
	\caption{Representation of univariate polynomials of degree $n$. 
	}
	\label{tb:Poly-1D}
	\centering
	\renewcommand{\arraystretch}{1.1} 
	\begin{tabular}{rcccc}
		\hline
		Degree $n$ & $L$ &  $\#\mbox{Weight}$ & $\#\mbox{Node}$ & $L^{\infty}$-Error\\
		\hline
		3        & 3 &  66         & 14       & 1.78e-15\\
		7        & 4 &  188        & 31		  & 1.78e-15\\
		15       & 5 &  429        & 64       & 4.44e-15\\
		31       & 6 &  910        & 129      & 5.33e-15\\
		63       & 7 &  1871       & 258      & 5.33e-15\\
		127      & 8 &  3792       & 515      & 5.33e-15\\		
		\hline 
	\end{tabular}
\end{table}	

Numerical tests for 2-dimensional polynomials in
tensor-product space and hyperbolic cross space are
presented in Tables \ref{tb:2D-1} and \ref{tb:2D-2},
respectively. The coefficients of corresponding power series
are all randomly generated according to standard normal
distribution. The results verify the stability and
efficiency of our method.

\begin{table}[!htbp]
	\caption{Representation of polynomials in tensor-product space $Q_N^2$.}
	\label{tb:2D-1}
	\centering
	%\footnotesize% fontsize
	\renewcommand{\arraystretch}{1.1}
	\begin{tabular}{rcccc}
		\hline
		Degree $N$ & $L$ &  $\#\mbox{Weight}$ & $\#\mbox{Node}$  & $L^{\infty}$-Error\\
		\hline
		3      & 5 &  378        & 64        & 1.11e-15\\
		7      & 7 &  1570       & 246       & 8.88e-15\\
		15     & 9 &   6376      & 988       & 1.60e-14\\
		31     & 11&  25758      & 4002      & 7.11e-14 \\
		63     & 13&  103668     & 16168     & 8.88e-14 \\
		\hline
	\end{tabular}
\end{table}

\begin{table}[!htbp]
	\caption{Representation of polynomials in hyperbolic cross polynomial space.}
	\label{tb:2D-2}
	\centering
	\renewcommand{\arraystretch}{1.1} 
	\begin{tabular}{rcccc}
		\hline
		Degree $N$ & $L$ &  $\#\mbox{Weight}$ & $\#\mbox{Node}$  & $L^{\infty}$-Error\\
		\hline
		7        & 7 &	1254       & 217       & 3.55e-15\\
		15       & 9 &  3277       & 554       & 1.24e-14\\
		31       & 11&  8022       & 1351       & 5.32e-14\\
		63       & 13&  19039       & 3196       & 2.24e-14\\
	   127       & 15&  44052       & 7393       & 4.26e-14\\
		\hline	
	\end{tabular}
\end{table}

Next, we present some results of approximated 1-dimensional
and 2-dimensional smooth functions using our approach, and
compare them with trained ReLU network approximations.  We
first show the results of approximating $\sin(x)$ using
ReQU network of our approach and ReLU network with randomly
initialized coefficients. The ReQU network is constructed
using proposed method based on a polynomial approximation of
degree $8$ and then trained by gradient descent method.  The
result is shown in the left plot of Fig. \ref{fig:sin1d}.
For the ReLU network approximation, we take 5 ReLU networks
with same structure (8 layers of hidden nodes with each
layer has 64 ReLU nodes, full connected) are trained using
mini-batch stochastic gradient descent method. The best
result among the 5 ReLU networks is shown in the right plot
of Fig. \ref{fig:sin1d}. Note that the number of hidden
nodes used by the ReQU network is less than $64$, and it
give much better results than the trained ReLU network. By
training the constructed ReQU network, the approximation
error can be further reduced.  Similar results for
approximating 2-dimensional function $\sin(x) \sin(
y)$ are presented in Fig. \ref{fig:sin2d}.

\begin{figure}[!htbp]
	\centering
	\includegraphics[width=0.47\textwidth]{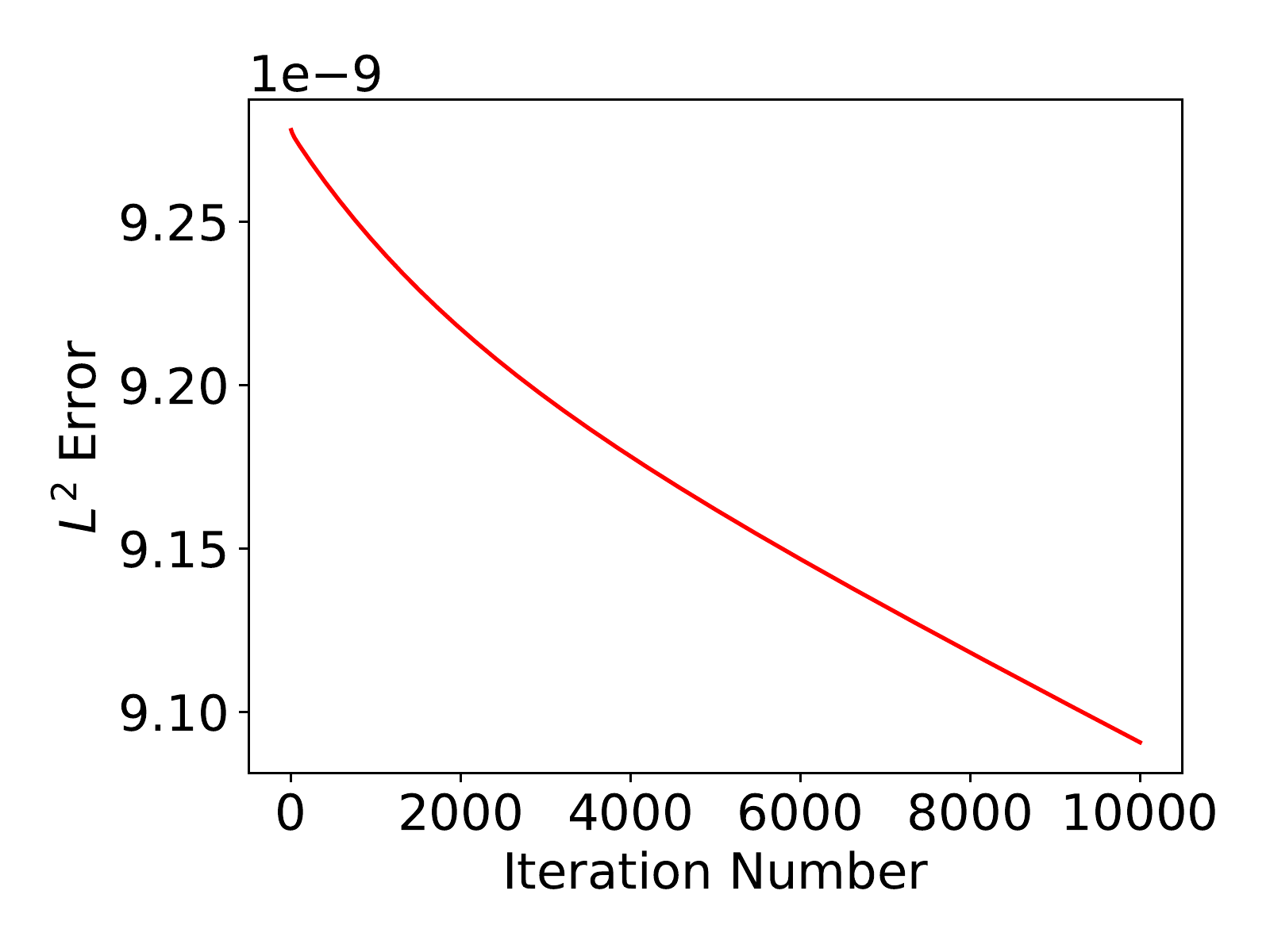}
	\includegraphics[width=0.45\textwidth]{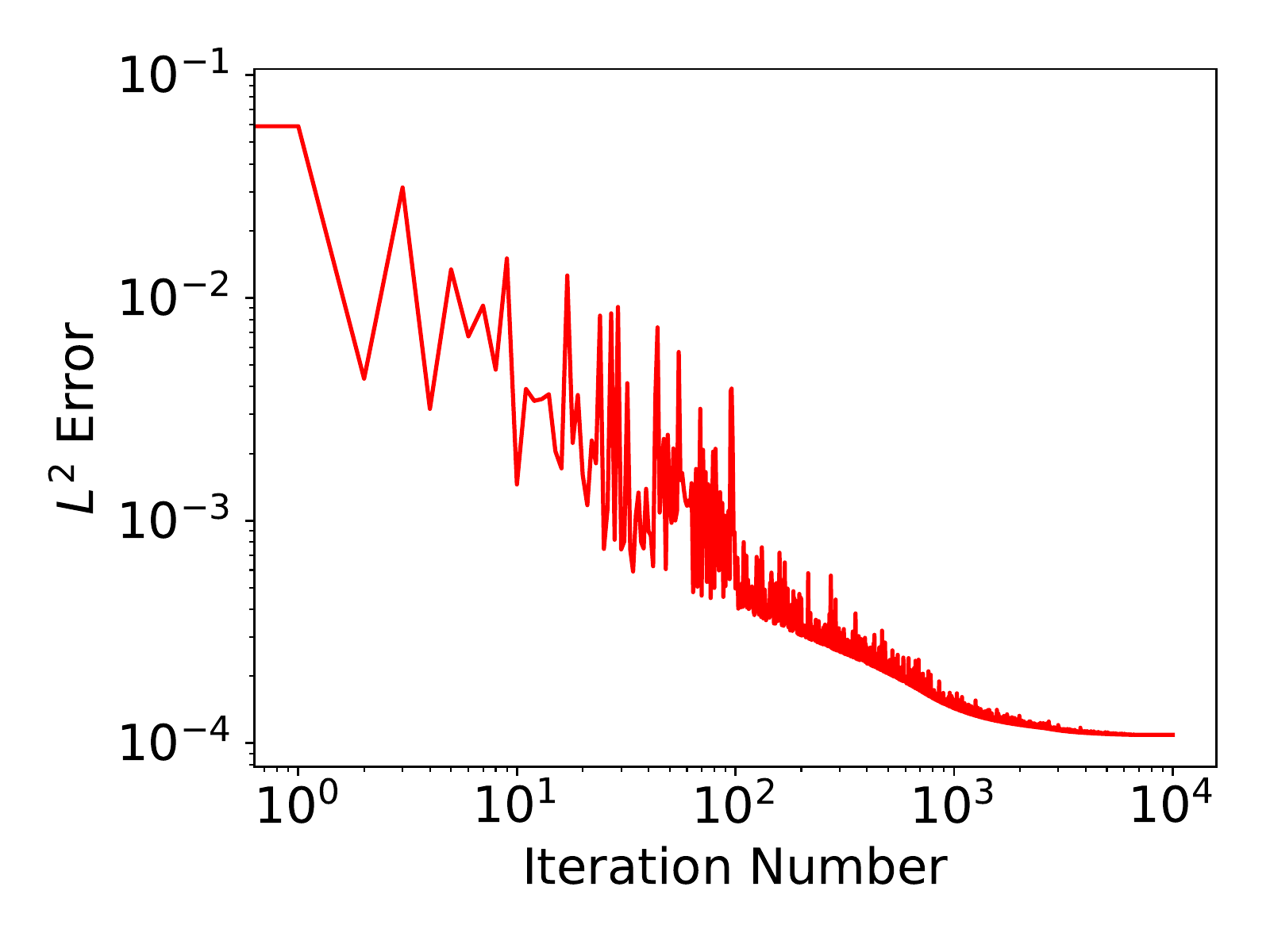}
	\caption{Approximating $\sin(x)$ function using ReQU
		and ReLU neural networks. Left: result of ReQU network
		initialized by polynomials of degree $8$ and then
		trained by a gradient descent method.  Right: result
		of ReLU network (8 fully connected hidden layers with each one has 64 ReLU nodes) with a random initialization
		and trained by a mini-batch gradient descent method.}
	\label{fig:sin1d}
\end{figure}   
\begin{figure}[!htbp]
	\centering
	\includegraphics[width=0.47\textwidth]{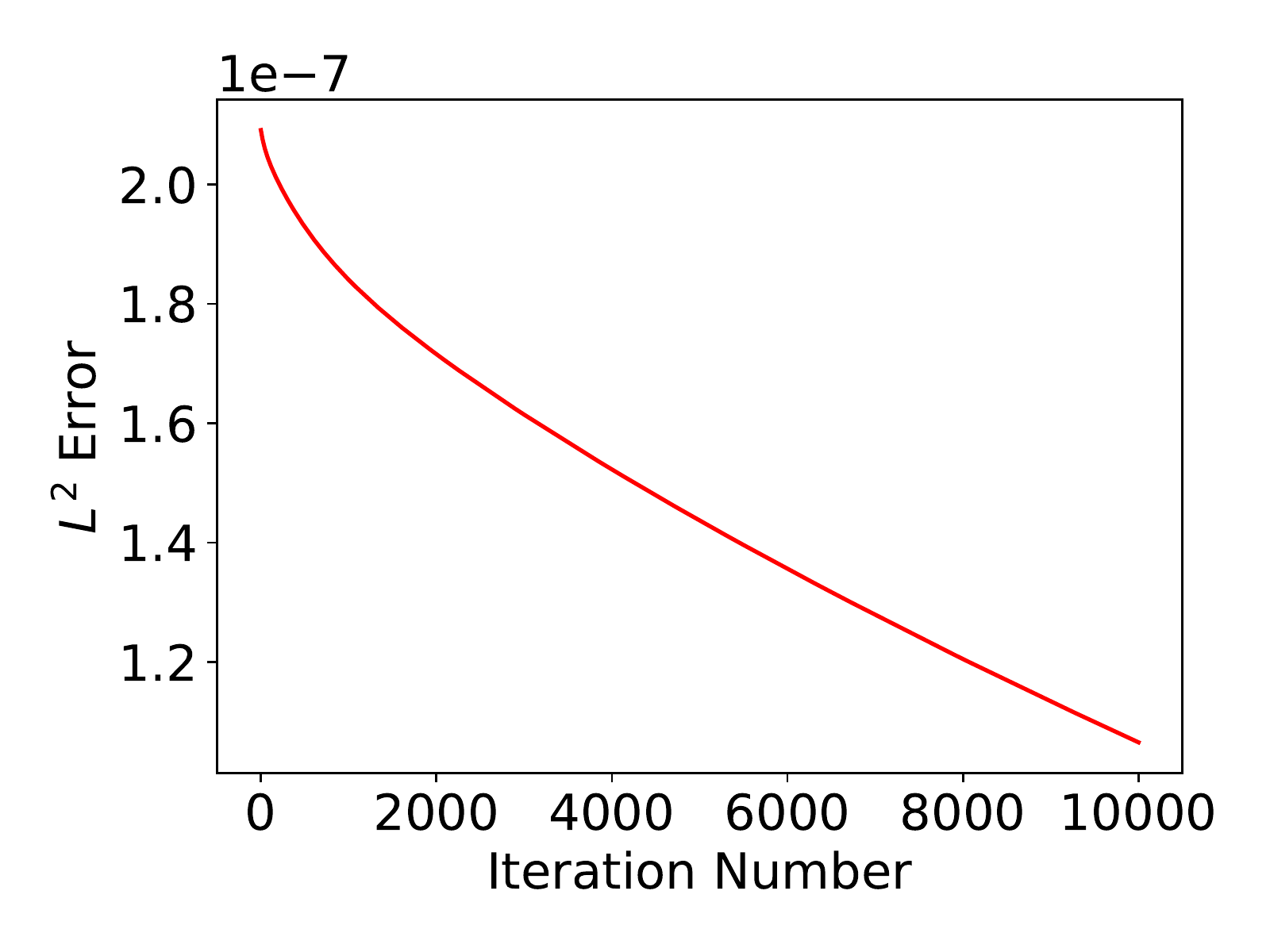}
	\includegraphics[width=0.45\textwidth]{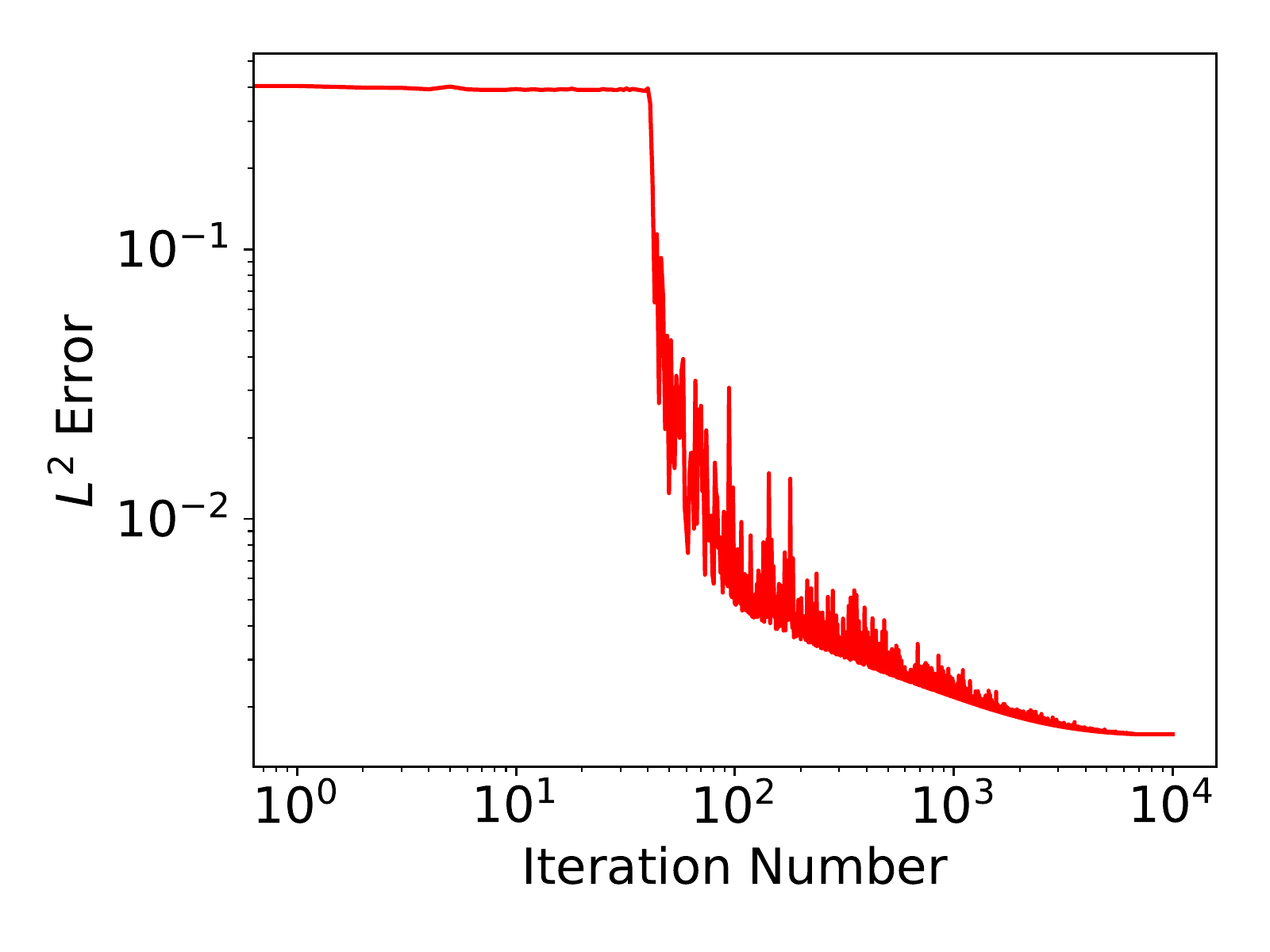}
	\caption{Approximating $\sin(x)\sin(y)$ using
		ReQU and ReLU neural networks. Left: result of ReQU
		network initialized by a 2-d polynomial in tensor-product space $Q_N^2 (N=9)$
		and then trained by a gradient descent method; Right:
		result of ReLU network (8 fully connected hidden layers, each one has 128 ReLU nodes) with a random
		initialization and then trained by a mini-batch
		gradient descent method.}
	\label{fig:sin2d}
\end{figure}    

\section{Conclusion and future work}

In this paper, we gave constructive proofs of some error bounds
for approximating smooth functions by deep neural networks
using RePU function as the activation functions.  The proofs
rely on the fact that polynomials can be represented by RePU
networks with no approximation error. We construct several
optimal algorithms for such representations, in which
polynomials of degree no more than $n$ are converted into a
ReQU network with $\mathcal{O}(\log_2 n)$ layers, and the
size of the network is of the same scale as the dimension of the polynomial space to be approximated. Then by using the
classical polynomial approximation theory, we obtain upper error bounds for ReQU networks approximating smooth functions,
which show clear advantages of using ReQU activation
function, comparing to the existing results for ReLU
networks. In general, the ReLU network required to
approximate a sufficient smooth function, is $\mathcal{O}(\log \frac1\varepsilon)$ times
larger than the corresponding ReQU network. Here
$\varepsilon$ is the approximation error. To achieve
$\varepsilon$-approximation for $f\in
B^{\infty}_{\alpha,\beta}$, the number of layer of ReQU
network required to obtain this approximation is
$\mathcal{O}(\log_2 \log\frac{1}{\varepsilon})$, while the
corresponding best known results is
$\mathcal{O}(\log\frac{1}{\varepsilon})$ for ReLU network.
For high dimensional functions with bounded mixed
derivatives, we give error bounds that have a weaker
exponentially dependence on $d$, by using hyperbolic
cross/sparse grid spectral approximation, in particular if
optimized hyperbolic cross polynomial projections are used,
there is no term related to $\varepsilon$ is exponentially dependent on $d$. Since only global polynomial approximations are considered in this paper, the results obtained also hold for deep
networks with non-rectified power units. The use of
rectified units gives the neural network the ability to
approximate piecewise smooth functions efficiently, which
will be analyzed in a separate paper.

Our constructions of RePU network also reveal the close
relation between the depth of the RePU network and the
``order'' of polynomial approximation.
The advantage of using {\em deep} over {\em shallow} neural
ReQU networks is clearly shown by our constructive proofs:
by using one hidden layer, a ReQU network can only represent
piecewise quadratic polynomials; by using $n$ hidden layers, a ReQU
network can represent piecewise polynomials of degree up to
$\mathcal{O}(2^n)$.  The ReQU networks we built for
approximating smooth functions all have a tree-like
structure, and are sparsely connected. This may give some
hints on how to design appropriate structures of neural
networks for some practical applications.

We have shown theoretically that for approximating
sufficient smooth functions, ReQU networks are superior to
ReLU networks in terms of approximation error. We also
present efficient and stable algorithm to construct ReQU
network based on polynomial approximation. Our preliminary
results demonstrated that our constructions are numerically
stable and efficient. The constructed neural network can be
regarded as a good initial of RePU network and further
trained to get better results.  For low dimensional
problems, this approach is much more accurate than the
results obtained by direct training a randomly initialized ReLU
neural networks.

In practical applications, the functions to be approximated
may have different kinds of non-smoothness, which are
problem dependent. The training method is another key factor
that affects the application of neural networks. We
will continue our study in these directions. In particular,
we will study the approximation error of piecewise smooth
functions with deep ReQU networks, and investigate whether
those popular training methods proposed to train ReLU
networks are efficient for training RePU
networks. Meanwhile, we will try deep RePU networks on some
practical problems where the underlying functions are
smooth, e.g.  minimum action methods for large PDE
systems\cite{wan_dynamicsolverconsistent_2017}, PDEs with
random coefficients\cite{musharbash_error_2015}, and moment
closure problem in complex fluid
\cite{yu_nonhomogeneous_2010} and turbulence
modeling\cite{mellor_development_1982}, etc.

%%%% Acknowledgments %%%%%%%%
\section*{Acknowledgments}
We are indebted to Prof. Jie Shen and Prof. Li-Lian Wang for
their stimulating conversations on spectral methods. We would like also
to think Prof. Christoph Schwab and Prof. Hrushikesh N. Mhaskar for providing us some related references. This work was partially supported by
China National Program on Key Basic Research Project 2015CB856003, NNSFC
Grant 11771439, 91852116, and China Science Challenge Project,
no. TZ2018001.
The computations were  performed on the PC clusters of
State Key Laboratory of Scientific and Engineering Computing of Chinese Academy of Sciences.

%\section*{Reference}
%\bibliographystyle{siamplain}
%\bibliographystyle{plain}
\bibliographystyle{unsrt}
\bibliography{sgDNN}

\end{document}